\documentclass[a4paper,centertags,oneside,12pt]{amsart}
\usepackage{mathrsfs}
\usepackage{appendix}
\usepackage{amssymb}
\usepackage{fancyhdr}
\usepackage{charter}
\usepackage{typearea}
\usepackage{pdfsync}
\usepackage{mathrsfs}
\usepackage[a4paper,top=3cm,bottom=3cm,left=3cm,right=3cm]{geometry}

\usepackage{enumitem}
\setlist[1]{itemsep=5pt}
\newcommand{\comment}[1]{}
\usepackage{color}
\makeatletter
      \def\@setcopyright{}
      \def\serieslogo@{}
      \makeatother

\newcommand{\field}[1]{\mathbb{#1}}
\newcommand{\C}{\field{C}}
\newcommand{\R}{\field{R}}
\newcommand{\N}{\field{N}}

\newcommand{\Complex}{\mathbb C}
\newcommand{\Real}{\mathbb R}
\newcommand{\ddbar}{\overline\partial}

\newcommand{\Td}{\widetilde}

\newcommand{\abs}[1]{\left\vert#1\right\vert}

\newcommand{\set}[1]{\left\{#1\right\}}
\newcommand{\To}{\rightarrow}

\newtheorem{theorem}{Theorem}[section]
\newtheorem{lemma}[theorem]{Lemma}

\newtheorem{proposition}[theorem]{Proposition}

\newtheorem{definition}[theorem]{Definition}

\numberwithin{equation}{section}
\begin{document}

\title[]{Szeg\H{o} kernel expansion and equivariant  embedding of CR manifolds with circle action}
\author[]{Hendrik Herrmann}
\address{Mathematical Institute, University of Cologne, Weyertal 86-90, 50931 Cologne, Germany}
\thanks{The first-named author was supported by the Graduiertenkolleg 1269: ``Global Structures in Geometry and Analysis'' and by the Institute of Mathematica, Academia Sinica}
\email{post@hendrik-herrmann.de or heherrma@math.uni-koeln.de}

\author[]{Chin-Yu Hsiao}
\address{Institute of Mathematics, Academia Sinica, 6F, Astronomy-Mathematics Building,
No.1, Sec.4, Roosevelt Road, Taipei 10617, Taiwan}
\thanks{The second-named author was partially supported by Taiwan Ministry of Science of Technology project
104-2628-M-001-003-MY2 and the Golden-Jade fellowship of Kenda Foundation.}
\email{chsiao@math.sinica.edu.tw or chinyu.hsiao@gmail.com}

\author[]{Xiaoshan Li}
\address{School of Mathematics
and Statistics, Wuhan University, Hubei 430072, China \& Institute of Mathematics, Academia Sinica, 6F, Astronomy-Mathematics Building,
No.1, Sec.4, Roosevelt Road, Taipei 10617, Taiwan}
\thanks{The third-named author was  supported by NSFC No. 11501422, Postdoctoral Science Foundation of China 2015M570660 and Central university research Fund 2042015kf0049}
\email{xiaoshanli@whu.edu.cn or xiaoshanli@math.sinica.edu.tw}

\setlength{\headheight}{14pt}

\begin{abstract}
Let $X$ be a compact strongly pseudoconvex CR manifold with a transversal CR $S^1$-action. In this paper, we establish the asymptotic expansion of Szeg\H{o} kernels of positive Fourier components and by using the asymptotics, we show that $X$ can be equivariant CR embedded into some $\mathbb C^N$ equipped with a simple \(S^1\)-action. An equivariant embedding of quasi-regular Sasakian manifold is also derived.
\end{abstract}

\maketitle

\section{Introduction}
Let $X$  be a compact
strongly pseudoconvex CR manifold.
The question of whether or not $X$ admits a CR embedding into
a complex Euclidean space has attracted a lot of attention.
By a theorem of  Boutet de Monvel~\cite{BdM1:74b}, any compact strongly pseudoconvex CR manifold of dimension greater than or equal to five can be CR embedded into $\mathbb C^N$ for some $N$. The classical example of non-embeddable three dimensional strongly
pseudoconvex CR manifold given by Rossi \cite{Ro65} (see also \cite{B79}) showed that an arbitrary real analytic deformation of the standard CR structure on the three-sphere may fail to be embeddable. There exists an extensive literature on the embaddability of deformed CR structures. For this subject, we refer the reader to \cite{BE90, Le92, Ep92, CCY1, CCY2} and the  references therein.
Given a compact strongly pseudoconvex CR manifold equipped with a locally free transversal CR $S^1$-action, it was shown in \cite{Le92, MY07} and also \cite{HM14I} that $X$ can always be CR embedded into some complex space. 

In this work, we attack the embedding problem for compact strongly pseudoconvex CR manifolds equipped with a transversal CR circle action from a pure analytic point of view. More precisely, we develop an asymptotic expansion for the Szeg\H{o} kernels (\(S_m\), see Section~\ref{subsec:SKDef}) concerning CR functions which lie in the space of positive Fourier components (\(H^0_{b,m}(X)\), see Definition \ref{def-16-09-01}).
Our main results involving that features are Theorem \ref{c2} and Theorem \ref{e2} which we deduce from a general result about Szeg\H{o} kernel expansion (see Theorem \ref{t1}) using partially the machinery developed by the second and third named authors in \cite{HL15, HL15a}. Theorem \ref{c2} describes the expansion on the regular part, i.e.~ the part of the manifold where \(S^1\) acts globally free. Here, the expansion works out well. Difficulties occur on the complement of the regular part, but we can still prove some expansion as stated in Theorem \ref{e2}. Roughly speaking, the problem is that the leading term of the expansion does not change smoothly (even not continuously) from the regular part to its complement.

Inspired by a work of the second named author  \cite{H14a}, our second main result in this paper is using Theorem \ref{c2} and Theorem \ref{e2} to construct a CR embedding which is equivariant with respect to a simple \(S^1\)-action on \(\C^N\) (see Theorem \ref{t-gue151127} ). Recently, Hsiao-Li-Marinescu \cite{HLM16} established the Kodaira embedding theorem for CR manifolds with circle action which complements the results of this paper with the study of the embedding in the presence of a positive line bundle but without the hypothesis of strict pseudoconvexity.

Before we state the precise embedding result, let us see some examples of
compact strongly pseudoconvex CR manifolds in $\C^N$ to get an idea how such simple \(S^1\)-actions could look like.

Example I: Let $X=\set{(z_1,z_2,\ldots,z_n)\in\Complex^n:\, \abs{z_1}^2+\abs{z_2}^2+\abs{z_3}^2+\cdots+\abs{z_n}^2=1}$ which is a  CR manifold with  a transversal CR
$S^1$-action (see Definition \ref{f1}):
\[e^{i\theta}\circ (z_1,z_2,\ldots,z_n)=(e^{im_1\theta}z_1,e^{im_2\theta}z_2,\ldots,e^{im_n\theta}z_n),\]
where $(m_1,\ldots,m_n)\in\mathbb N^n$.

Example II: $X=\set{(z_1,z_2,z_3)\in\Complex^3:\, \abs{z_1}^2+\abs{z_2}^2+\abs{z_3}^2+\abs{z^2_1+z_2}^4+\abs{z^3_2+z_3}^6=1}$. Then $X$ admits a transversal CR locally free $S^1$-action:
\[e^{i\theta}\circ (z_1,z_2,z_3)=(e^{i\theta}z_1,e^{2i\theta}z_2,e^{6i\theta}z_3).\]
\begin{definition}\label{d-gue151127}
We say that an $S^1$-action $e^{i\theta}$ on $\mathbb C^N$ is simple if
\[e^{i\theta}\circ(z_1,\ldots,z_N)=(e^{im_1\theta}z_1,\ldots,e^{im_N\theta}z_N),\ \ \forall (z_1,\ldots,z_N)\in\mathbb C^N,\ \ \forall \theta\in[0,2\pi),\]
where  $(m_1,\ldots,m_N)\in\mathbb N^N$. The minimal weight of a simple $S^1$-action  on $\mathbb C^N$ is given by \(\min\limits_{1\leq j\leq N}m_j\).
\end{definition}
Our goal is to prove the following equivariant embedding theorem for compact strongly pseudoconvex CR manifolds which admit a transversal CR $S^1$-action.

\begin{theorem}\label{t-gue151127}
Let $(X,T^{1,0}X)$ be a compact strongly pseudoconvex CR manifold with a transversal CR locally free $S^1$-action $e^{i\theta}$. Then, for any \(m_0\in\N\) we can find \(N\in\N\), a simple \(S^1\)-action on  \(\C^N\) with minimal weight bigger than \(m_0\) and an equivariant CR embedding
\begin{equation*}\Phi:X\To\mathbb C^N,~~
	x\mapsto(\Phi_1(x),\ldots,\Phi_N(x)).
\end{equation*}
More precisely, \(\Phi\) is an embedding, \(\Phi_1,\ldots,\Phi_N\) are CR functions and there exist {\tt \(m_1,\ldots,m_N\in \N\)}, \(m_j>m_0\) for \(j=1,\ldots,N\),
such that
\[\Phi(e^{i\theta}\circ x)=(e^{im_1\theta}\Phi_1(x),\ldots,e^{im_N\theta}\Phi_N(x))
=e^{i\theta}\circ\Phi(x),\ \ \forall x\in X,\ \ \forall\theta\in[0,2\pi),\]
where the simple $S^1$-action on $\mathbb C^N$ is given by $e^{i\theta}\circ (z_1, \cdots, z_N)=(e^{im_1\theta}z_1, \cdots, e^{im_N\theta}z_N).$

\end{theorem}
The equivariance, where a lower bound for the minimal weight of the simple \(S^1\)-action on the embedding space can be chosen arbitrarily large, distinguishes Theorem~\ref{t-gue151127} from the embedding results mentioned at the beginning of this section.

An application of Theorem \ref{t-gue151127} is an equivariant embedding result for quasi-regular Sasakian manifolds. Sasakian manifolds have gained prominence in physics and algebraic geometry, especially in String theory (see \cite{BG06, BGS08}).
Let \(X\) be a compact smooth manifold of dimension $2n-1, n\geq2$. The triple  $(X, g,\alpha)$ where \(g\) is a Riemannian metric and \(\alpha\) a real 1-form is called Sasakian manifold if  the cone $\mathcal{C}(X)=\{(x,t)\in X\times \R_{>0}\}$ is a K\"ahler manifold  with complex struture $J$ and K\"ahler form \(t^2d\alpha+2tdt\wedge\alpha\) which are compatible with the metric $t^2g +dt\otimes dt$ (see  \cite{Bl76}, \cite{BGS08}, \cite{Ge97}, \cite{OV06}, \cite{OV07}). As a consequence, $X$ is a compact strongly pseudoconvex CR manifold and the Reeb vector field $\xi$, defined by \(\alpha(\cdot)=g(\xi,\cdot)\), induces a transversal CR action on $X$. If the orbits of $\xi$ are closed, the Sasakian structure is called quasi-regular. In this case, the Reeb vector field generates a locally free transversal CR $S^1$-action on $X$. We thus can identify a compact quasi-regular Sasakian manifold with a compact strongly pseudoconvex CR manifold $(X,T^{1,0}X)$ with a transversal CR locally free $S^1$-action where the induced vector field of the $S^1$-action coincides with the Reeb vector field on $X$. From Theorem \ref{t-gue151127}, we get
\begin{theorem}
Let $X$ be a quasi-regular Sasakian manifold which admits a transversal CR $S^1$-action induced by the Reeb vector field. For any \(m_0\in\N\), there is an equivariant CR embedding of $X$ into some $\mathbb C^N$ equipped with a simple \(S^1\)-action with minimal weight bigger than \(m_0\).
\end{theorem}

The paper is organized as follows. In the rest of this section we will give an outline of the idea of the proof of Theorem~\ref{t-gue151127} and introduce some terminology. Section~\ref{sec1} contains the proofs of the results on Szeg\H{o} kernel expansion for positive Fourier components. The embedding result (Theorem \ref{t-gue151127}) is proven in Section~\ref{sec3}.

\subsection{The idea of the proof of Theorem~\ref{t-gue151127}}\label{subsec:IdeaOfProof} We refer the reader to Section~\ref{st}, Section~\ref{st-I} and Section~\ref{st-II} for some notations and terminology used here. Assume that $(X, T^{1,0}X)$ is a compact connected strongly pseudoconvex CR manifold of dimension $2n-1, n\geq 2$, with a transversal CR locally free $S^1$-action $e^{i\theta}$. Let $T$ denote the vector field induced by the $S^1$-action and let $\overline\partial_b$ be the tangential Cauchy-Riemann operator on $X$. For every $m\in\mathbb N$, let $H^0_{b,m}(X):=\{u\in C^\infty(X): \overline\partial_bu=0, Tu=imu\}$ be the $m$-th positive Fourier component  of the space of global smooth CR functions. The main inspiration of this paper is the following: In \cite{HL15a} the second and third-named author have shown  that ${\rm dim}H^0_{b, m}(X)\thickapprox m^{n-1}$ as $m\rightarrow\infty$. Hence, the space of CR functions which lie in the positive Fourier components is very large and we therefore ask whether $X$ can be CR embedded into some \(\C^N\) by CR functions which lie in the positive Fourier components. In this work we give an affirmative answer to this question and as a corollary, we deduce Theorem~\ref{t-gue151127}. More precisely, we will prove
\begin{theorem}\label{t-gue151127b}
Let $X$ be a compact connected strongly pseudoconvex CR manifold with a locally free transversal CR $S^1$-action. Then $X$ can be CR embedded into some complex space by the CR functions which lie in the positive Fourier components.\end{theorem}
Motivated by the second-named author's work on the Kodaira embedding theorem (\cite{H14a}, \cite{H15}, \cite{H14b}), we will  use the asymptotic expansion of the Szeg\H{o}  kernel with respect to $H^0_{b, m}(X)$ to prove Theorem~\ref{t-gue151127}.
For every $k\in\mathbb N$, let $X_k$ and $X_{\rm reg}$ be defined as in (\ref{e-gue150614}).
Let $\{f_j\}_{j=1}^{d_{m}}\subset H^0_{b, m}(X)$ be an orthonormal basis. The $m$-th Szeg\H{o}  kernel $S_m(x,y)$ is given by $S_m(x,y):=\sum^{d_m}_{j=1}f_j(x)\overline{ f_j(y)}$.
Let us first consider \[\begin{split}\Psi^1_m:X\To\mathbb C^{d_m}, x\mapsto(f_1(x),\ldots,f_{d_m}(x)).\end{split}\] We notice that $S_m(x,y)=0$ on $X_k$ if $k\nmid m$. From this observation, we see that $X\setminus X_{\rm reg}\neq\emptyset$ implies that $\Psi^1_m$ can not be an embedding even if $m$ is large. Suppose $X=X_1\cup X_2\cup\cdots\cup X_l$. For $1\leq k\leq l$, let $\{f_j^k\}_{j=1}^{d_{km}}$ be an orthonormal basis of $H^0_{b, km}(X)$. We next consider the map  \[\begin{split}\Psi_m:X&\To\mathbb C^{\tilde N_m},\\ x&\mapsto(f^1_1(x),\ldots,f^1_{d_m}(x),f^2_1(x),\ldots,f^2_{d_{2m}}(x),\ldots,f^l_1(x), \ldots,f^l_{d_{lm}}(x)),\end{split}\]where $\tilde N_m=d_m+d_{2m}+\cdots+d_{lm}$.
In Section~\ref{Sec: asymp exp on Xreg}, we will show that on a canonical coordinate patch $D\subset X_{{\rm reg}}$ with canonical coordinates $x=(z,\theta)$, we have\begin{equation}\label{e-gue151127g}\begin{split}&S_m(x, y)\equiv\frac{1}{2\pi}e^{im(x_{2n-1}-y_{2n-1}+\Phi(z, w))}\hat b(z, w, m)\mod O(m^{-\infty}),\\&\hat b(z, w,m)\sim\sum^\infty_{j=0}m^{n-1-j}\hat b_j(z,w),\hat b_j(z, w)\in C^{\infty}(D\times D),\ \  j=0, 1, 2, \cdots,\\&\hat b_0(z, z)\neq0\end{split}\end{equation}(see Theorem~\ref{c2}). Moreover, for fixed $x_0\in X_k$, $k>1$, one finds that $k\nmid m$ implies $S_m(x, x_0)=0$ and that $k\mid m$ leads to \begin{equation}\label{e-gue151127gI}S_m(x, x_0)\equiv \frac{k}{2\pi}e^{im(x_{2n-1}+\Phi(z, 0))}\hat b(z, 0, m)\mod O(m^{-\infty})\end{equation} for some canonical coordinate patch $D$ with canonical coordinates $x=(z,\theta)$, $x_0\in D$, $(z(x_0),\theta(x_0))=(0,0)$ (see Theorem~\ref{e2}). It should be mentioned that \eqref{e-gue151127g} and \eqref{e-gue151127gI} are based on Boutet de Monvel-Sj\"ostrand's classical result on Szeg\H{o}  kernels~\cite{BS76} (after the seminal work ~\cite{Fe74} of Fefferman) and the complex stationary phase formula of Melin-Sj\"ostrand~\cite{MS75}. From \eqref{e-gue151127g} and \eqref{e-gue151127gI}, we can check that $\Psi_m$ is an immersion if $m$ is large. But $\Psi_m$ is not globally injective: in general, assume that $m$ is even, then we can not separate the points $p\in X_k$ and $e^{i\frac{\pi}{k}}\circ p$, where $k>1$.  To overcome this difficulty, let $\{g_j^k\}_{j=1}^{d_{k(m+1)}}$
be an orthonormal basis of $ H^0_{b, k(m+1)}(X)$ for $1\leq k\leq l$. For  any $k$, $1\leq k\leq l$, we define a CR map from $X$ to Euclidean space as follows
\begin{equation*}
\Phi^k_m: X\rightarrow \mathbb C^{d_{km}+d_{k(m+1)}}, x\mapsto (f_1^k(x),\cdots, f_{d_{km}}^k(x), g_1^k(x), \cdots, g_{d_{k(m+1)}}^k(x) ),
\end{equation*}
and we set
\begin{equation*}
\Phi_m: X\rightarrow \mathbb C^{N_m}, x\rightarrow (\Phi^1_m(x), \cdots, \Phi^l_m(x)),\end{equation*}
where $N_m=\sum\limits_{k=1}^l (d_{km}+d_{k(m+1)}).$  We thus try to prove that $\Phi_m$ is globally injective. It is not difficult to see that $\Phi_m$ can separate the points $p\in X_k$ and $e^{i\theta}\circ p$, where $p\neq e^{i\theta}\circ p$,  if $m$ is large enough. But another \emph{difficulty} comes from the fact that the expansion \eqref{e-gue151127g} converges only locally uniformly on $X_{{\rm reg}}$ and on $X\setminus X_{{\rm reg}}$, we can only get an expansion of $S_m(x,x_0)$ for fixed $x_0\in X\setminus X_{{\rm reg}}$ and this causes that $\Phi_m$ could not be globally injective. To overcome this difficulty, we analyze carefully the behavior of the Szeg\H{o}  kernel $S_m(x,y)$ near the complement of $X_{\rm reg}$ and in Section~\ref{s-f}, we can construct many CR functions $h_1,\ldots,h_K$ with large potentials near the complement of $X_{\rm reg}$ which lie in the positive Fourier components such that the map \[x\in X\To (\Phi_m(x),h_1(x),\ldots,h_K(x))\in\mathbb C^{N_m+K}\] is an embedding if $m$ is large (see Theorem~\ref{main1}). This finishes the proof of Theorem~\ref{t-gue151127b}.

\subsection{Set up and terminology}\label{st}

Let $(X, T^{1,0}X)$ be a compact connected orientable CR manifold of dimension $2n-1, n\geq 2$, where $T^{1,0}X$ is the CR structure of $X$. That is $T^{1,0}X$ is a subbundle of the complexified tangent bundle $\mathbb{C}TX$ of rank $n-1$, satisfying $T^{1,0}X\cap T^{0,1}X=\{0\}$, where $T^{0,1}X=\overline{T^{1,0}X}$ and $[\mathcal V,\mathcal V]\subset\mathcal V$, where $\mathcal V=C^\infty(X, T^{1,0}X)$.

We assume that $X$ admits an $S^1$-action: $S^1\times X\rightarrow X, (e^{i\theta}, x)\rightarrow e^{i\theta}\circ x$. Here we use $e^{i\theta}$ to denote the $S^1$-action. Set $X_{\rm reg}=\{x\in X: \forall e^{i\theta}\in S^1, ~\text{if}~e^{i\theta}\circ x=x,~ \text{then}~e^{i\theta}=\rm id\}$. For every $k\in\mathbb N$, put
\begin{equation}\label{e-gue150614}
X_k:=\set{x\in X: e^{i\theta}\circ x\neq x, \forall\theta\in(0,\frac{2\pi}{k}), e^{i\frac{2\pi}{k}}\circ x=x}.\end{equation}
Thus, $X_{\rm reg}=X_1$. In this paper, for simplicity we always assume that $X_{1}\neq\emptyset.$ Actually, one can re-normalize the $S^1$-action by lifting such that the new $S^1$-action satisfies  $X_1\neq\emptyset$, (see \cite[Remark 1.14]{CHT15}).

Let $T\in C^\infty(X, TX)$ be the global real vector field induced by the $S^1$-action given as follows
\begin{equation}\label{Eq: IndVec}
(Tu)(x)=\frac{\partial}{\partial\theta}\left(u(e^{i\theta}\circ x)\right)\Big|_{\theta=0},~u\in C^\infty(X).
\end{equation}
\begin{definition}\label{f1}
We say that the $S^1$-action $e^{i\theta} ~(0\leq\theta<2\pi$) is CR if
$$[T, C^\infty(X, T^{1,0}X)]\subset C^\infty(X, T^{1,0}X),$$
where $[~,~ ]$ is the Lie bracket between the smooth vector fields on $X$.
Furthermore, we say that the $S^1$-action is transversal if for each $x\in X$ one has
$$\mathbb CT(x)\oplus T_x^{1,0}(X)\oplus T_x^{0,1}X=\mathbb CT_xX.$$
\end{definition}
We throughout assume that $(X, T^{1,0}X)$ is a compact connected CR manifold with a transversal CR locally free $S^1$-action and we denote by $T$ the global vector field induced by the $S^1$-action. Let $\omega_0\in C^\infty(X,T^*X)$ be the global real one form uniquely determined by $\langle\,\omega_0\,,\,u\,\rangle=0$, for every $u\in T^{1,0}X\oplus T^{0,1}X$ and $\langle\,\omega_0\,,\,T\,\rangle=-1$.

We recall
\begin{definition}\label{d-1.2}
For $x\in X$, the Levi-form $\mathcal L_x$ associated with the CR structure is the Hermitian quadratic form on $T_x^{1,0}X$ defined as follows. For any $U, V\in T_x^{1,0}X$, pick $\mathcal U, \mathcal V\in C^\infty(X, T^{1,0}X)$ such that $\mathcal U(x)=U, \mathcal V(x)=V$. Set
\begin{equation*}
\mathcal L_x(U, \overline V)=\frac{1}{2i}\langle[\mathcal U, \overline{\mathcal V}](x), \omega_0(x)\rangle
\end{equation*}
where $[~,~]$ denotes the Lie bracket between smooth vector fields. Note that $\mathcal L_x(U, \overline V)$ does not depend on the choice of $\mathcal U$ and $\mathcal V$.
\end{definition}
\begin{definition}\label{Def: spsc}
The CR structure on $X$ is called pseudoconvex at $x\in X$ if $\mathcal L_x$ is semi-positive definite. It is called strongly pseudoconvex at $x$ if $\mathcal L_x$ is positive definite. If the CR structure is (strongly) pseudoconvex at every point of $X$, then $X$ is called a (strongly) pseudoconvex CR manifold.
\end{definition}
Denote by $T^{\ast 1,0}X$ and $T^{\ast0,1}X$ the dual bundles of
$T^{1,0}X$ and $T^{0,1}X$, respectively. Define the vector bundle of $(0,q)$-forms by
$\Lambda^qT^{\ast0,1}X$. Let $D\subset X$ be an open subset. Let $\Omega^{0,q}(D)$
denote the space of smooth sections of $\Lambda^qT^{\ast0, 1}X$ over $D$.

Fix $\theta_0\in [0, 2\pi)$. Let
$$d e^{i\theta_0}: \mathbb CT_x X\rightarrow \mathbb CT_{e^{i\theta_0}x}X$$
denote the differential map of $e^{i\theta_0}: X\rightarrow X$. By the properties of transversal CR $S^1$-actions, we can check that
\begin{equation}\label{a}
\begin{split}
de^{i\theta_0}:T_x^{1,0}X\rightarrow T^{1,0}_{e^{i\theta_0}x}X,\\
de^{i\theta_0}:T_x^{0,1}X\rightarrow T^{0,1}_{e^{i\theta_0}x}X,\\
de^{i\theta_0}(T(x))=T(e^{i\theta_0}x).
\end{split}
\end{equation}
Let $(e^{i\theta_0})^\ast: \Lambda^q(\mathbb CT^\ast X)\rightarrow\Lambda^q(\mathbb CT^\ast X)$ be the pull back of $e^{i\theta_0}, q=0,1\cdots, n-1$. From (\ref{a}), we can check that for every $q=0, 1,\cdots, n-1$
\begin{equation}\label{j1}
(e^{i\theta_0})^\ast: \Lambda^qT^{\ast0,1}_{e^{i\theta_0}x}X\rightarrow \Lambda^qT_x^{\ast0,1}X.
\end{equation}

Let $u\in\Omega^{0,q}(X)$ be a section. The Lie derivative of $u$ along the direction $T$ is denoted by $Tu$.
From (\ref{j1}) we have $Tu\in\Omega^{0, q}(X)$ for all $u\in\Omega^{0, q}(X)$.

Let $\overline\partial_b:\Omega^{0,q}(X)\rightarrow\Omega^{0,q+1}(X)$ be the tangential Cauchy-Riemann operator. From (\ref{a}), it is straightforward to deduce
\begin{equation}\label{c}
T\overline\partial_b=\overline\partial_bT~\text{on}
~\Omega^{0,q}(X).
\end{equation}

For every $m\in\mathbb Z$, put $\Omega^{0,q}_m(X):=\{u\in\Omega^{0,q}(X): Tu=imu\}$. We denote by $\overline\partial_{b, m}$ the restriction of $\overline\partial_b$ to $\Omega^{0, q}_m(X)$.
From (\ref{c}) we have the $\ddbar_{b, m}$-complex for every $m\in\mathbb Z$:
\begin{equation}\label{e-gue140903VI}
\ddbar_{b, m}:\cdots\To\Omega^{0,q-1}_m(X)\To\Omega^{0,q}_m(X)\To\Omega^{0,q+1}_m(X)\To\cdots.
\end{equation}
For $m\in\mathbb Z$, the $q$-th $\ddbar_{b, m}$-cohomology is given by
\begin{equation}\label{a8}
H^{q}_{b,m}(X):=\frac{{\rm Ker\,}\ddbar_{b}:\Omega^{0,q}_m(X)\To\Omega^{0,q+1}_m(X)}{\operatorname{Im}\ddbar_{b}:\Omega^{0,q-1}_m(X)\To\Omega^{0,q}_m(X)}.
\end{equation}
Moreover,  we have (see Theorem 1.13 in \cite{HL15a})
\begin{equation}\label{a1}
{\rm dim} H^q_{b, m}(X)<\infty, ~\text{for all}~ q=0, \ldots, n-1.\end{equation}

\begin{definition}\label{def-16-09-01}
A function $u\in C^\infty(X)$ is a Cauchy-Riemann function (CR function for short)
if $\overline\partial_bu=0$, that is $\overline Zu=0$ for all $Z\in C^{\infty}(X, T^{1, 0}X)$. For $m\in \mathbb N$, $H^0_{b, m}(X)$ is called the $m$-th positive Fourier component of the space of CR functions.
\end{definition}

\subsection{Hermitian CR geometry}\label{st-I}

\begin{definition}\label{Def: T-rigid metric}
Let $D$ be an open set and let $V\in C^\infty(D, \mathbb CTX)$ be a vector field over $D$. We say that $V$ is rigid if
\begin{equation*}
de^{i\theta_0}(V(x))=V(e^{i\theta_0}x)
\end{equation*}
holds for any $x\in D$ and $\theta_0\in[0,2\pi)$ satisfying  $ e^{i\theta_0}\circ x\in D.$
\end{definition}

\begin{definition}
Let $\langle\cdot|\cdot\rangle$ be a Hermitian metric on $\mathbb CTX$.
We say that $\langle\cdot|\cdot\rangle$ is rigid if for all rigid vector fields $V, W$ on $D$, where $D$ is any open set, we have
\begin{equation*}
\langle V(x)|W(x)\rangle=\langle (de^{i\theta_0}V)(e^{i\theta_0}\circ x)|(de^{i\theta_0}W)(e^{i\theta_0}\circ x)\rangle, \forall x\in D, \theta_0\in[0,2\pi).
\end{equation*}
\end{definition}

\begin{lemma}[Theorem 9.2 in \cite{H14b}]\label{a7}
Let $X$ be a compact CR manifold with a transversal CR $S^1$-action. There always exists a rigid Hermitian metric $\langle\cdot|\cdot\rangle$ on $\mathbb CTX$ such that $T^{1,0}X\bot T^{0,1}X,$ $ T\bot(T^{1,0}X\oplus T^{0,1}X),$ $ \langle T|T\rangle=1$ and $\langle u|v\rangle$ is real if $u, v$ are real tangent vectors.
\end{lemma}
From now on, we fix a rigid Hermitian metric $\langle\cdot|\cdot\rangle$ on $\mathbb CTX$ satisfying all the properties in Lemma~\ref{a7}.
The Hermitian metric $\langle\cdot|\cdot\rangle$ on $\mathbb CTX$ induces by duality a Hermitian metric on $\mathbb CT^\ast X$ and also on the bundles of $(0,q)$-forms $\Lambda^q T^{\ast 0, 1}X, q=0,1\cdots,n-1.$ We will denote all these induced
metrics by $\langle\cdot|\cdot\rangle$. For every $v\in \Lambda^q T^{\ast0, 1}X$, we write
$|v|^2:=\langle v|v\rangle$. We have the pointwise orthogonal decompositions:
\begin{equation*}
\begin{split}
&\mathbb CT^{\ast}X=T^{\ast1,0}X\oplus T^{\ast0,1}X\oplus\{\lambda\omega_0:\lambda\in\mathbb C\},\\
&\mathbb CTX=T^{1,0}X\oplus T^{0,1}X\oplus\{\lambda T:\lambda\in\mathbb C\}.
\end{split}
\end{equation*}

For  $p\in X$, locally there is an orthonormal frame $\{U_1, \ldots, U_{n-1}\}$ of $T^{1, 0}X$  such that the Levi-form $\mathcal L_p$ is diagonal with respect to this frame. That is, $\mathcal L_p(U_i, \overline{U_j})=\lambda_j\delta_{ij}$, where $\delta_{ij}=1$ if $i=j$, $\delta_{ij}=0$ if $i\neq j$. The entries $\{\lambda_1, \ldots, \lambda_{n-1}\}$ are called the eigenvalues of the Levi-form at $p$ with respect to the rigid Hermitian metric $\langle\cdot|\cdot\rangle$. Moreover, the determinant of $\mathcal L_p$ is defined by $\det\mathcal L_p=\lambda_1(p)\cdots\lambda_{n-1}(p)$.
\subsection{Canonical local coordinates}\label{q}\label{st-II}

In this work, we need the following result due to Baouendi-Rothschild-Treves, (see \cite{BRT85}).

\begin{theorem}\label{j}
Let $X$ be a compact CR manifold of ${\rm dim}X=2n-1, n\geq2$ with a transversal CR $S^1$-action. Let $\langle\cdot|\cdot\rangle$ be a rigid Hermitian metric on $X$ as in Lemma \ref{a7}.
For $x_0\in X$, there exists a local patch $D$ and local coordinates $(x_1,\cdots,x_{2n-1})=(z,\theta)=(z_1,\cdots,z_{n-1},\theta), z_j=x_{2j-1}+ix_{2j}, 1\leq j\leq n-1, x_{2n-1}=\theta$, centered at $x_0$. In terms of these coordinates,
$D=\{(z, \theta)\in\mathbb C^{n-1}\times\mathbb R: |z|<\varepsilon, |\theta|<\delta\}$ and on $D$
\begin{equation}\label{e-can}
\begin{split}
&T=\frac{\partial}{\partial\theta}\\
&Z_j=\frac{\partial}{\partial z_j}+i\frac{\partial\varphi(z)}{\partial z_j}\frac{\partial}{\partial\theta},j=1,\cdots,n-1,
\end{split}
\end{equation}
where $\{Z_j(x)\}_{j=1}^{n-1}$ form a basis of $T_x^{1,0}X$, for each $x\in D$ and $\varphi(z)\in C^\infty(D,\mathbb R)$ is independent of $\theta$. Moreover, on $D$, $\varphi(z)=\sum\limits_{j=1}^{n-1}\lambda_j|z_j|^2+O(|z|^3), \forall (z, \theta)\in D$, where $\{\lambda_j\}_{j=1}^{n-1}$ are the eigenvalues of the Levi-form of $X$ at $x_0$ with respect to the given rigid Hermitian metric on $X$. We call $D$ a canonical local patch and $(z, \theta, \varphi)$ canonical coordinates centered at $x_0$.
\end{theorem}
In order to make use of the \(S^1\)-action locally, we need coordinates which contains nearly whole orbits. On the regular part, this is possible without any restrictions while on the irregular part we just find such coordinates approximately. This issue is stated in the following two lemmas.
\begin{lemma}[\cite {HL15a}, Lemma 1.17]\label{l-gue150615}
Fix $x_0\in X_{\rm reg}$. Then we can find canonical coordinates $(z,\theta,\varphi)$ centered at $x_0$ and defined on a canonical local patch $D=\{(z,\theta): |z|<\varepsilon_0, |\theta|<\pi\}$.
\end{lemma}
\begin{lemma}[\cite{HL15a}, Lemma 1.18 ]\label{l-gue150616}
Let $x_0\in X_k$ be a point, where $k\in\mathbb N$, $k>1$. For every $\epsilon>0$, $\epsilon$ small, we can find canonical coordinates $(z,\theta,\varphi)$ centered at $x_0$ and defined on a canonical local patch $D_\epsilon=\{(z,\theta): |z|<\varepsilon_0, |\theta|<\frac{\pi}{k}-\epsilon\}$.
\end{lemma}
\begin{lemma}[\cite{HL15a}, Lemma 1.19]\label{b5}
Fix $x_0\in X$. Let $(z, \theta, \varphi)$ be canonical coordinates centered at $x_0$ and defined on a canonical chart $D=\tilde D\times (-\delta, \delta).$  We denote by $dv_X$ the volume form associated with the rigid Hermitian metric.
Then on $D$ one has
$dv_X=\lambda(z)dv(z)d\theta$
with $\lambda(z)\in C^\infty(\tilde D, \mathbb R)$ which does not depend on $\theta$ and $dv(z)=2^{n-1}dx_1\cdots dx_{2n-2}$.
\end{lemma}

\section{Szeg\H{o}  kernel expansion}\label{sec1}

\subsection{Some standard notations}
First, we introduce some standard notations and definitions.
We shall use the following notations: $\mathbb N_0=\mathbb N\cup\{0\}.$ An element $\alpha=(\alpha_1, \cdots, \alpha_n)\in \mathbb N_0^n$ will be called a multiindex and the size of $\alpha$ is $|\alpha|=\alpha_1+\cdots+\alpha_n.$ We write $x^{\alpha}=x_1^{\alpha_1}\cdots x_n^{\alpha_n}$, $x=(x_1, \cdots, x_n),$ $\partial_x^\alpha=\partial_{x_1}^{\alpha_1}\cdots\partial_{x_n}^{\alpha_n}$,
$\partial_x^{\alpha}=\frac{\partial^{|\alpha|}}{\partial x^\alpha}$. Let $z=(z_1, \cdots, z_n)$, $z_j=x_{2j-1}+i x_{2j}$, $j=1, \cdots, n$ be the coordinates of $\mathbb C^n$. We write $z^\alpha=z_1^{\alpha_1}\cdots z_n^{\alpha_n}, $ $\overline z^{\alpha}=\overline z_1^{\alpha_1}\cdots \overline z_n^{\alpha_n}$, $\frac{\partial^{|\alpha|}}{\partial z^\alpha}=\partial_{z_1}^{\alpha_1}\cdots\partial_{z_n}^{\alpha_n}$, $\frac{\partial ^{|\alpha|}}{\partial\overline z^\alpha}=\partial^{\alpha_1}_{\overline z_1}\cdots\partial^{\alpha_n}_{\overline z_n}.$

In this section, we will study the semi-classical asymptotic expansion of the Szeg\H{o}  kernel of positive Fourier components. We recall some notations in semi-classical analysis.

\begin{definition}\label{g1}
Let $W$ be an open subset of $\mathbb R^N$.  Let $S(1; W)=S(1)$ be the set of $a\in C^\infty(W)$ such that for every $\alpha\in \mathbb N_0^N$, there exists constant $C_{\alpha}$ such that $|\partial^\alpha_x a(x)|\leq C_\alpha$ on $W$.  If $a=a(x, k)$ depends on $k\in (1, \infty)$, we say that $a(x, k)\in S_{\rm loc}(1; W)=S_{\rm loc}(1)$ if $\chi(x)a(x, k)$ is uniformly bounded in $S(1)$ when $k$ varies in $(1, \infty)$ for every $\chi(x)\in C_0^\infty (W).$ For $m\in\mathbb R$, we put $S^m_{\rm loc}(1; W)=S^m_{\rm loc }(1)=k^m S_{\rm loc}(1)$. If $a\in S^{m_0}_{\rm loc}(1)$, $a_j\in S^{m_j}_{\rm loc}(1), m_j\searrow-\infty$, we say that $a\thicksim\sum_{j=0}^{\infty}a_j$ in $S^{m_0}_{\rm loc}(1)$ if $a-\sum_{j=0}^{N_0}a_j\in S^{m_{N_0+1}}_{\rm loc}(1)$ for every $N_0$.
\end{definition}


Let $W_1, W_2$ be two open subsets of $\mathbb R^N.$ If $A: C_0^\infty(W_1)\rightarrow\mathcal D^\prime(W_2)$ is continuous, by the Schwartz kernel theorem (Theorem 5.2.1 in \cite{Ho03}) we write $K_A(x, y)$ or $A(x, y)$ to denote the distribution kernel of $A$. The following two statements are equivalent

(a) $A$ can be extended to an continuous operator : $\mathcal E^\prime (W_1)\rightarrow C^\infty(W_2)$,

(b) $A(x, y)\in C^\infty(W_1\times W_2)$.\\
If $A$ satisfies (a) or (b), we say that $A$ is smoothing.


A $k$-dependent continuous operator $A_k: C_0^\infty(W_1)\rightarrow\mathcal D^{\prime}(W_2)$ is called $k$-negligible  if $A_k$ is smoothing and the kernel $A_k(x, y)$ of $A_k$ satisfies
$|\partial^\alpha_x\partial^\beta_y A_k(x, y)|=O(k^{-m})$ locally uniformly on every compact set in $W_1\times W_2$, for all multi-indices $\alpha, \beta\in\mathbb N_0^N$ and all $m\in\mathbb N_0$.  Let $C_k: C_0^\infty(W_1)\rightarrow\mathcal D^\prime(W_2)$ be another $k$-dependent continuous operator. We write $A_k\equiv C_k\mod O(k^{-\infty})$  or $A_k(x, y)\equiv C_k(x, y)\mod O(k^{-\infty})$ if $A_k-C_k$ is $k$-negligible. We write $A_k=C_k+O(k^{-\infty})$ if $A_k\equiv C_k\mod O(k^{-\infty})$. Similarly, we write $B_k(x)\equiv 0 \mod O(k^{-\infty})$ for any $k$-dependent smooth function $B_k(x)\in C^\infty(W)$ if $|\partial_x^\alpha B_k(x)|=O(k^{-m})$   locally uniformly on every compact subset of $W$ for all $\alpha\in \mathbb N_0^N$ and all $m\in\mathbb N_0$.

\subsection{Asymptotic Szeg\H{o}  kernel expansion}\label{subsec:SKDef}


Let $(\,\cdot\,|\,\cdot\,)$ be the inner product on $\Omega^{0,0}(X)$ induced by $dv_X$. Let $L^2(X)$ (resp.~$L^2_m(X)$) be the completions of $\Omega^{0,0}(X)$ (resp.~$\Omega^{0,0}_m(X)$) with respect to $(\,\cdot\,|\,\cdot\,)$. By elementary Fourier analysis one has $L^2_m(X)\perp L^2_{m^\prime}(X)$ for $m\neq m^\prime, m, m^\prime\in\mathbb Z.$ For $m\in\mathbb Z$, let $Q_m: L^2(X)\rightarrow L^2_m(X)$ be the orthogonal projection with respect to $(\,\cdot\,|\,\cdot\,)$.

From now on we assume $m\in\mathbb N$. Let $S_m: L^2(X)\rightarrow H^0_{b, m}(X)$ be the orthogonal projection with respect to $(\,\cdot\,|\,\cdot\,)$. We call $S_m$ the $m$-th Szeg\H{o}  projection. From (\ref{a1}), one finds ${\rm dim}H^0_{b, m}(X)<\infty$. Let $\{f_j\}_{j=1}^{d_m}$ be an orthonormal basis of $H^0_{b, m}(X)$. Then the $m$-th Szeg\H{o}  kernel function is given by $S_m(x)=\sum_{j=1}^{d_m}|f_j(x)|^2.$ Let $S_m(x, y)$ be the distribution kernel with respect to the operator $S_m$ which is given by $S_m(x, y)=\sum_{j=1}^{d_m}f_j(x)\overline{f_j(y)}.$ The goal of this section is to study the semi-classical asymptotic expansion of $S_m(x, y)$.

We extend $\overline\partial_b$ to $L^2(X)$ in the sense of distribution and denote its kernel by ${\rm Ker}(\overline\partial_b)=\{u\in L^2(X): \overline\partial_bu=0\}$ which is a closed subspace of $L^2(X)$. Let $S: L^2(X)\rightarrow{\rm Ker}(\overline\partial_b)$ be the usual Szeg\H{o}  projection. We denote by $S(x, y)$ the distribution kernel of the Szeg\H{o} projection.

\begin{lemma}\label{a2}
With the notations above, we have
\begin{equation}\label{f2}
H^0_{b, m}(X)={\rm Ker}(\overline\partial_b)\cap L_m^2(X)
\end{equation} and
\begin{equation}\label{f3}
S_m u=SQ_m u=Q_mSu, ~\forall u\in C^\infty(X).
\end{equation}
\end{lemma}

\begin{proof}
It is obvious that $H^0_{b, m}(X)\subset {\rm Ker}(\overline\partial_b)\cap L^2_m(X).$  The converse is a direct corollary from following subelliptic estimate (see theorem 1.12 in \cite{HL15})
\begin{equation}\label{f4}
\|u\|_s\leq C_{s, m}(\|\overline\partial_b u\|_{s-1}+\|u\|), \forall u\in H^s(X)\cap L^2_m(X), s\geq 1,
\end{equation}
where $H^s(X)$ is the usual Sobolev space on $X$, $\|u\|_s$ is the usual Sobolev norm of order $s$ and $C_{s, m}$ is a constant.

For any $u\in C^\infty(X)$, write $u=u_1+u_2$, $u_1\in H^0_{b, m}(X), u_2\in {H^0_{b, m}(X)}^\perp$. For any $v\in H^0_{b, m}(X)$, we have
\begin{equation*}
(S_mu|v)=(u_1|v)=(u|v)=(Q_mu|v)=(SQ_mu|v).
\end{equation*}
For any $v\in L^2(X)\bigcap H^0_{b,m}(X)^\perp$, we have
\[(S_mu|v)=0=(SQ_mu|v)\]
since $S_mu, SQ_mu\in H^0_{b,m}(X)$.
This implies  $S_mu=SQ_mu$ for all $u\in C^\infty(X)$. Similarly, we have $S_mu=Q_m Su$ for all $u\in C^\infty(X)$.
\end{proof}

Fix $x_0\in X$. Let $(z, \theta, \varphi)$ be canonical coordinates centered at $x_0$ and defined on a canonical local patch
$D_1=\{(z, \theta): |z|<\varepsilon_1, |\theta|<\delta_1\}$. Choose $D=\{(z, \theta): |z|<\varepsilon, |\theta|<\delta\}\Subset D_1.$ 

Choose two cut-off functions $\chi, \chi_1\in C_0^\infty(D_1)$ such that $\chi=1$ in some small neighborhood of $\overline D$ and $\chi_1=1$ in some small neighborhood of ${\rm supp}\chi$. By Lemma \ref{a2}, we have $S_m =SQ_m $ and hence
\begin{equation*}
\chi S_m=\chi SQ_m=\chi S\chi_1Q_m +\chi S(1-\chi_1)Q_m.
\end{equation*}
We write $F=\chi S(1-\chi_1)$ and $F_m=\chi S(1-\chi_1)Q_m$ and denote by $F(x, y)$ and $F_m(x, y)$ the distribution kernels of $F$ and $F_m$, respectively. We will show
\begin{lemma}\label{f5}
$F_m: C_0^\infty (D)\rightarrow \mathcal E^\prime(D_1)$ is $m$-negligible.
\end{lemma}
\begin{proof}
Since $\rm supp~\chi\cap \rm supp ~(1-\chi_1)=\emptyset$, by a result of Boutet de Monvel-Sj\"ostrand \cite{BS76} (see also \cite{H08} and \cite{HM14}) we know that $F$ is smoothing.
Let $\cup_{j=1}^{n_0} U_j$ be a finite covering of $X$. We assume that all the $U_js , 1\leq j\leq n_0$ are canonical local patches. Choose a partition of unity $\{\rho_j\}_{j=1}^{n_0}$ with ${\rm supp} \rho_j\Subset U_j$,  $1\leq j\leq n_0 ,$ and $\sum^{n_0}_{j=1}\rho_j=1$ on $X$. Then for all $u\in C_0^\infty (D)$ we have
\begin{equation}\label{j3}
F_mu=FQ_mu=F\left(\sum_{j=1}^{n_0} \rho_j Q_m u\right)=\sum_{j=1}^{n_0}F(\rho_jQ_m u).
\end{equation}
For $1\leq j\leq n_0$, let $y=(w, y_{2n-1})$ be  canonical coordinates in $U_j$. Then on $U_j$ one finds
\begin{equation*}
\rho_jQ_m u=\rho_j(y)(Q_m u)(y)=\rho_j(y)\hat u_m(w)e^{im y_{2n-1}}.
\end{equation*}
Set $F_j(x, y)=F(x, y)\rho_j(y)$ for $x\in D, y\in U_j$. Then on $D$, by a direct calculation we have
\begin{equation}\label{j4}
\begin{split}
F(\rho_j Q_m u)(x)
=-\frac{1}{2\pi mi}\int_{U_j}\left(\int^{2\pi}_0\frac{\partial F_j}{\partial y_{2n-1}}(x, e^{i\theta}\circ y)e^{im\theta}d\theta\right)u(y)\lambda(w)dv(w) dy_{2n-1}.
\end{split}
\end{equation}
By (\ref{j3}), (\ref{j4}) and the induction method, we have $F_m(x, y)=O(m^{-N})$ locally uniformly for all $N\in \mathbb N$ and similarly for the derivatives. Thus, the lemma follows.
\end{proof}

Set $G=\chi S\chi_1$ and $G_m=\chi S\chi_1Q_m$. Write $D_1=\tilde D_1\times (-\delta_1, \delta_1)$ and $D=\tilde D\times (-\delta, \delta)$ with $\tilde D_1=\{z\in\mathbb C^{n-1}: |z|<\varepsilon_1\}$ and $\tilde D=\{z\in\mathbb C^{n-1}: |z|<\varepsilon\}$. Assume that on $D_1$, $\chi_1(y)=\tilde\chi_1(w)\tilde\chi_2(y_{2n-1})$ holds with $\tilde\chi_1(w)\in C_0^\infty(\tilde D_1), \tilde\chi_2(y_{2n-1})\in C_0^\infty(-\delta_1, \delta_1)$ and $\tilde\chi_1(w)=1$ in some small neighborhood of $\overline{\tilde D}$ and $\tilde \chi_2=1$ in some small neighborhood of $[-\delta,\delta]$. Let $u\in C^\infty_0(D)$ be a smooth function. On $D_1$, we write $(Q_mu)(y)=\hat u_m(w)e^{imy_{2n-1}}$, $\hat u_m(w)\in C^\infty(\tilde D_1)$. Then on $D$ we have
\begin{equation}\label{a6}
\begin{split}
G_m u(x)
&=\int_{\tilde D_1}\tilde\chi_1(w)\hat u_m(w)\lambda(w)\Bigr(\int_{-\delta_1}^{\delta_1}\chi(x)S(x, w, y_{2n-1})\tilde\chi_2(y_{2n-1})e^{im y_{2n-1}}dy_{2n-1}\Bigr)dv(w).
\end{split}
\end{equation}
In order to calculate the integral with respect to $dy_{2n-1}$ in  (\ref{a6}), we need the following  result due to Boutet de Monvel and Sj\"ostrand \cite{BS76}, \cite{ H08} and Hsiao-Marinescu \cite{HM14}.
\begin{theorem}\label{t1}
Let $X$ be a compact strongly pseudoconvex CR manifold with a transversal CR $S^1$-action. For any $x_0\in X$, let $D_1$ be the canonical local patch defined as in Theorem \ref{j} with canonical coordinates $(z, \theta, \varphi)$ centered at $x_0$. Then on $D_1\times D_1$ the distribution kernel $S(x, y)$ of the Szeg\H{o}  projection $S: L^2(X)\rightarrow {\rm Ker}(\overline\partial_b)$ satisfies
\begin{equation}\label{FIO}
S(x, y)=\int_{0}^\infty e^{i\Psi(x, y)t}b(x, y, t)dt
\end{equation}
in the sense of oscillatory integrals, where
\begin{equation}\label{j6}
\begin{split}
&\Psi(x, y)\in C^{\infty}(D_1\times D_1),
\Psi(x, y)=x_{2n-1}-y_{2n-1}+\Phi(z, w),\\
& \Phi(z, w)=-\overline\Phi(w, z), \exists~ c>0: {\rm Im}\Phi\geq c|z-w|^2, \Phi(z, w)=0\Leftrightarrow z=w, \\
&\Phi(z, w)=i(\varphi(z)+\varphi(w))-2i\sum_{|\alpha|+ |\beta|\leq N}\frac{\partial^{|\alpha|+|\beta|}\varphi}{\partial z^\alpha\partial\overline z^\beta}(0)\frac{z^\alpha}{\alpha !}\frac{\overline w^\beta}{\beta!}+O(|(z, w)|^{N+1}), \forall N\in\mathbb N_0,\\
&b(x, y, t)\sim \sum_{k=0}^\infty b_k(x, y)t^{n-1-k}~{\rm in}~ S^{n-1}_{\rm loc}(1; D_1\times D_1),\\
&b_j(x, y)\in C^\infty(D_1\times D_1), j=0, 1, \cdots,\\
&b_0(x,x)=\frac{1}{2}\pi^{-n}\abs{\det\mathcal L_x},\ \ \forall x\in D_1.
\end{split}
\end{equation}

\end{theorem}
By Theorem \ref{t1},  the integral with respect to $dy_{2n-1}$ in  (\ref{a6}) can be computed by making use of the stationary phase formula due to Melin-Sj\"orstrand \cite{MS75}. Substituting (\ref{FIO}) and (\ref{j6}) to (\ref{a6}) and changing coordinates $t=m\sigma$, where $m\in\N$ and $\sigma\in\R_{+}$, we have
\begin{equation}\label{e7}
\begin{split}
&\int_{-\delta_1}^{\delta_1}\chi(x)S(x, w, y_{2n-1})\tilde\chi_2(y_{2n-1})e^{im y_{2n-1}}dy_{2n-1}\\
=&m\int_{-\delta_1}^{\delta_1}\int_0^\infty e^{im[(x_{2n-1}-y_{2n-1})\sigma+\Phi(z, w)\sigma+y_{2n-1}]}\chi(x)b(x, y, m\sigma)\tilde\chi_2(y_{2n-1})d\sigma dy_{2n-1}.
\end{split}
\end{equation}
Set $$\tilde \Psi(x, w, y_{2n-1}, \sigma)=(x_{2n-1}-y_{2n-1})\sigma+\Phi(z, w)\sigma+y_{2n-1}.$$ Then one finds

\begin{equation*}
\frac{\partial \tilde\Psi}{\partial \sigma}=x_{2n-1}-y_{2n-1}+\Phi(z, w), \frac{\partial\tilde\Psi}{\partial y_{2n-1}}=-\sigma+1.
\end{equation*}
For any fixed $(x, w)$ the critical point of $\tilde \Psi$ is denoted by $x_c=(y_{2n-1}, \sigma)=(x_{2n-1}+\Phi(z, w), 1)$ which is the solution of the equations $\frac{\partial \tilde\Psi}{\partial \sigma}=0, \frac{\partial\tilde\Psi}{\partial y_{2n-1}}=0 $.
Moreover, the Hessian of $\tilde\Psi$ with respect to variables $(y_{2n-1}, \sigma)$  at the critical point $x_c$ is
\begin{equation*}
\left(
  \begin{array}{cc}
    \frac{\partial^2\tilde\Psi}{\partial\sigma\partial\sigma}  & \frac{\partial^2\tilde\Psi}{\partial\sigma\partial y_{2n-1}} \\
    \frac{\partial^2\tilde\Psi}{\partial y_{2n-1}\partial\sigma } & \frac{\partial^2\tilde\Psi}{\partial y_{2n-1}\partial y_{2n-1}} \\
  \end{array}
\right)\Big|_{x_c}=\left(
          \begin{array}{cc}
            0 & -1 \\
            -1 & 0 \\
          \end{array}
        \right)
\end{equation*}
which implies that $\tilde \Psi(x, w, y_{2n-1}, \sigma)$ is a non-degenerate complex valued phase function for any fixed $(x, w)$ in the sense of Melin and Sj\"ostrand \cite{MS75}. Hence, one can apply the stationary phase formula of Melin and Sj\"ostrand \cite{MS75} to carry out the $d\sigma dy_{2n-1}$ integration in (\ref{e7}):
\begin{equation}\label{a3}
\begin{split}
&m\int_{-\delta_1}^{\delta_1}\int_{0}^\infty e^{im\tilde\Psi(x, w, y_{2n-1}, \sigma)}\chi(x)b(x, y, m\sigma)\tilde\chi_2(y_{2n-1})d\sigma dy_{2n-1}\\
=&m\int_{-\delta_1}^{\delta_1}\int e^{im\tilde\Psi}\tau(\sigma)\chi(x)b(x, y, m\sigma)\tilde\chi_2(y_{2n-1})d\sigma dy_{2n-1}\\
+&m\int_{-\delta_1}^{\delta_1}\int e^{im\tilde\Psi}(1-\tau(\sigma))\chi(x)b(x, y, m\sigma)\tilde\chi_2(y_{2n-1})d\sigma dy_{2n-1},
\end{split}
\end{equation}
where $\tau(\sigma)\in C^\infty_0(\mathbb R)$ with ${\rm supp}\tau\Subset (\frac12, \frac32)$ and $\tau=1$ near $\sigma=1$.

First we show that on $D_1\times\tilde D_1$,  the second term on the right-hand side of (\ref{a3}) satisfies the following identity
\begin{equation}\label{a5}
m\int_{-\delta_1}^{\delta_1}\int e^{im\tilde\Psi(x, w, y_{2n-1}, \sigma)}(1-\tau(\sigma))\chi(x)b(x, y, m\sigma)\tilde\chi_2(y_{2n-1})d\sigma dy_{2n-1}\equiv 0 \mod O(m^{-\infty}).
\end{equation}
This is a direct corollary of the following formula
\begin{equation*}
e^{im\tilde\Psi}=\frac{1}{im(1-\sigma)}\frac{\partial}{\partial y_{2n-1}} e^{im\tilde\Psi}
\end{equation*}
and the integration by parts with respect to  the variable $y_{2n-1}$. For convenience we denote by $H_m(x, w)$ the left-hand side of (\ref{a5}).

Making use of Melin-Sj\"ostrand's stationary phase formula \cite{MS75}, the first term on the right-hand side of (\ref{a3}) becomes
\begin{equation}\label{a4}
\begin{split}
&m\int_{-\delta_1}^{\delta_1}\int e^{im\tilde\Psi}\tau(\sigma)\chi(x)b(x, y, m\sigma)\tilde\chi_2(y_{2n-1})d\sigma dy_{2n-1}\\
\equiv &e^{im(x_{2n-1}+\Phi(z, w))}\chi(x)\hat b(x, w, m)\mod O(m^{-\infty}),
\end{split}
\end{equation}
where
\begin{equation}\label{e-gue151124}
\begin{split}
&\hat b(x, w, m)\sim \sum_{j=0}^\infty \hat b_j(x, w)m^{n-1-j}~\text{in}~S^{n-1}_{\rm loc}(1; D_1\times\tilde D_1), \\
&\hat b_j(x, w)\in C^\infty(D_1\times\tilde D_1), j=0, 1, 2,\cdots.
\end{split}
\end{equation}
In particular, one has
\begin{equation}\label{e-gue151124I}
\begin{split}
\hat b_0(x, w)=(2\pi) \tilde b_0(x, w, x_{2n-1}+\Phi(z, w)),~
\hat b_0(x,z)=\pi^{1-n}\abs{{\rm det}\mathcal{L}_x},
\end{split}
\end{equation}
where $\tilde b_0$ denotes an almost analytic extension of $b_0$, that is $\tilde b_0(\Td x,\Td y)\in C^\infty(U_1\times U_1)$ with $\tilde b_0|_{D_1\times D_1}=b_0$ and $\abs{\ddbar_{\Td x}\tilde b_0(\Td x,\Td y)}+\abs{\ddbar_{\Td y}\tilde b_0(\Td x,\Td y)}\leq C_N(\abs{{\rm Im\,}\Td x}^N+\abs{{\rm Im\,}\Td y}^N)$, for every $N>0$ where $C_N>0$ is a constant. Here $U_1$ is an open set in $\Complex^{2n-1}$ with $U_1\bigcap\Real^{2n-1}=D_1$ (we identify $D_1$ with an open set in $\Real^{2n-1}$) and $\Td x$, $\Td y$ are complex coordinates of $\Complex^{2n-1}$.
Substituting (\ref{a5}) and (\ref{a4}) to (\ref{a6}) one has
\begin{equation}\label{h1}
\begin{split}
G_mu=&\int_{\tilde D_1}\tilde\chi_1(w)\hat u_m(w)e^{im(x_{2n-1}+\Phi(z, w))}\chi(x)\hat b(x, w, m)\lambda(w)dv(w)\\
+&\int_{\tilde D_1}\tilde\chi_1(w)\hat u_m(w)H_m(x, w)\lambda(w)dv(w)
\end{split}
\end{equation}
with $H_m(x, w)\equiv 0\mod O(m^{-\infty})$ on $D_1\times \tilde D_1.$

Choose $\eta(y_{2n-1})\in C_0^\infty(-\delta_1, \delta_1)$ such that $\int_{-\delta_1}^{\delta_1}\eta(y_{2n-1})dy_{2n-1}=1.$ Then the first term on the right-hand side of (\ref{h1}) is equal to
\begin{equation}\label{b1}
\begin{split}
&\int_{D_1}(Q_mu)(y)\tilde\chi_1(w)
\eta(y_{2n-1})e^{im(x_{2n-1}-y_{2n-1}+\Phi(z, w))}\chi(x)\hat b(x, w, m)\lambda(w)dw dy_{2n-1}\\
=&\chi(x)\int_{D_1}(Q_{-m} B_m)(x, y) u(y)\lambda (w)dy
=\chi(x)\int_{D}(Q_{-m} B_m)(x, y) u(y)\lambda (w)dy.
\end{split}
\end{equation}
Here, we have set
\begin{equation}\label{a9}
B_m(x, y)=e^{im(x_{2n-1}-y_{2n-1}+\Phi(z, w))}\hat b(x, w, m)\tilde\chi_1(w)\eta(y_{2n-1})
\end{equation}
and $(Q_{-m}B_m)(x,y)$ denotes that $Q_{-m}$ acts on $B_m(x,y)$ seen as a function in the variable $y$.
Combining (\ref{h1}), (\ref{b1}), (\ref{a9}) and Lemma~\ref{f5}, we have
\begin{equation*}
S_m(x, y)=\frac{1}{2\pi}\int_{-\pi}^{\pi}B_m(x, e^{i\theta}\circ y)e^{im\theta}d\theta+A_m(x, y), \forall x, y\in D\times D,
\end{equation*}
where $A_m(x, y)\equiv 0\mod O(m^{-\infty})$.
On the other hand, we have
$
S_m(x, y)=\sum_{j=1}^{d_m}f_j(x)\overline {f_j(y)}.
$
On $D$, we can write $f_j(x)=\hat f_j(z)e^{im x_{2n-1}}$ which leads to
$
S_m(x, y)=\sum_{j=1}^{d_m}\hat f_j(z)\overline {\hat f_j(w)}e^{im(x_{2n-1}-y_{2n-1})}.
$
Thus,
\begin{equation}\label{h6}
e^{-im x_{2n-1}}S_m(x, y)=\sum_{j=1}^{d_m}\hat f_j(z)\overline {\hat f_j(w)}e^{im(-y_{2n-1})}
\end{equation}
does not depend on $x_{2n-1}$. We get
\begin{equation}\label{h2}
e^{-im x_{2n-1}}S_m(x, y)=\frac{1}{2\pi}\int_{-\pi}^{\pi}e^{-im x_{2n-1}}B_m(x, e^{i\theta}\circ y)e^{im\theta}d\theta+e^{-im x_{2n-1}}A_m(x, y).
\end{equation}
Choose $\chi_0(x_{2n-1})\in C_0^\infty(-\delta, \delta)$ such that $\int_{-\delta}^{\delta}\chi_0(x_{2n-1})dx_{2n-1}=1.$  From (\ref{h6}) and (\ref{h2}) we have
\begin{equation}\label{h3}
\begin{split}
e^{-im x_{2n-1}}S_m(x, y)
=\frac{1}{2\pi}\int_{-\pi}^{\pi}\int_{-\delta}^{\delta}
&\chi_0(x_{2n-1})e^{-imx_{2n-1}}B_m(x, e^{i\theta}\circ y)e^{im\theta}dx_{2n-1}d\theta\\
&+C_m(z, y),
\end{split}
\end{equation}
where $C_m(z, y)=\int_{-\delta}^{\delta}A_m(x, y)e^{-im x_{2n-1}}\chi_0 (x_{2n-1})dx_{2n-1}$, $C_m(z,y)\equiv0\mod O(m^{-\infty})$.
Set \begin{equation}\label{h4}
\hat S_m(x, y)=e^{im x_{2n-1}}\int_{-\delta_1}^{\delta_1}\chi_0(x_{2n-1})e^{-im x_{2n-1}}B_m(x, y)dx_{2n-1}.
\end{equation}
From (\ref{a9}),(\ref{h2}), (\ref{h3}) and (\ref{h4}) we have

\begin{theorem}\label{b6}
Let $X$ be as in Theorem \ref{t1}. Consider the orthogonal projection $S_m: L^2(X)\rightarrow H^0_{b, m}(X)$. We denote by $S_m(x, y)$ the distribution kernel of $S_m$. For $x_0\in X$, let $(z, \theta, \varphi)$ be canonical coordinates centered at $x_0$ and defined on a canonical local patch $D_1=\{(z, \theta): |z|<\varepsilon_1, |\theta|<\delta_1\}$. For any $D=\{(z, \theta): |z|<\varepsilon, |\theta|<\delta\}\Subset D_1$, we have
\begin{equation*}
S_m(x, y)\equiv \frac{1}{2\pi}\int_{-\pi}^{\pi}\hat S_m(x, e^{i\theta}\circ y)e^{im\theta}d\theta\mod O(m^{-\infty})
\end{equation*}
on $D\times D$, where
\begin{equation}\label{b2}
\begin{split}
&\hat S_m(x, y)=e^{im(x_{2n-1}-y_{2n-1}+\Phi(z, w))}\hat b(z, w, m)\tilde\chi_1(w)\eta(y_{2n-1}),\\
&\Phi(z, w)=i(\varphi(z)+\varphi(w))-2i\sum_{|\alpha|+ |\beta|\leq N}\frac{\partial^{|\alpha|+|\beta|}\varphi}{\partial z^\alpha\partial\overline z^\beta}(0)\frac{z^\alpha}{\alpha !}\frac{\overline w^\beta}{\beta!}+O(|(z, w)|^{N+1}),\\
&\hat b(z, w, m)\sim \sum_{k=0}^\infty m^{n-1-k}\hat b_k(z, w)~{\rm in}~S^{n-1}_{\rm loc}(1; \tilde D\times \tilde D), \tilde D=\{z\in\mathbb C^{n-1}: |z|<\varepsilon\}, \\
&\hat b_0(z, w)=(2\pi)\int_{-\delta}^{\delta} \tilde b_0(z,x_{2n-1}, w, x_{2n-1}+\Phi(z, w))\chi_0(x_{2n-1})dx_{2n-1},\\
&\hat b_0(z, z)=\pi^{-(n-1)}\abs{\det\mathcal L_x},\ \ x=(z,0),\ \ \forall z\in\tilde D,
\end{split}
\end{equation}
and
\begin{equation*}
\begin{split}
&\hat b_j(z, w)\in C^\infty(\tilde D\times\tilde D), \forall j; \chi_0(x_{2n-1})\in C_0^\infty(-\delta, \delta),  \int_{-\delta}^{\delta}\chi_0(x_{2n-1})dx_{2n-1}=1;\\
&\chi_1(w)\in C_0^\infty(\tilde D_1), \ \ \mbox{$\chi_1=1$ in a neighborhood of $\overline{\tilde D}$};\\
&\eta(y_{2n-1})\in C_0^\infty(-\delta_1, \delta_1), \int_{-\delta_1}^{\delta_1}\eta(y_{2n-1})dy_{2n-1}=1.
\end{split}
\end{equation*}
Here, $\tilde b_0$ is as in \eqref{e-gue151124I}.
\end{theorem}
\subsection{Asymptotic Szeg\H{o}  kernel expansion on $X_{\rm reg}$ }\label{Sec: asymp exp on Xreg}

Recall that $X_{\rm reg}$ is the regular part of $X$, that is $X_{\rm reg}=\{x\in X: \forall e^{i\theta}\in S^1, ~\text{if}~e^{i\theta}\circ x=x,~ \text{then}~e^{i\theta}=\rm id\}$, and that we assume $X_{\rm reg}\neq\emptyset$. Fix $x_0\in X_{\rm reg}$. Let $(z, \theta, \varphi)$ be canonical coordinates centered at $x_0$ and defined on a canonical patch $D_1=\{(z, \theta): |z|<\varepsilon_1, |\theta|<\pi\}$.
Set $D=\{(z, \theta)\in\mathbb C^{n-1}\times\mathbb R: |z|<\varepsilon, |\theta|<\frac{\pi}{2}\}$ with $\varepsilon<\varepsilon_1$. From Theorem \ref{b6} it follows that
\begin{equation}\label{b3}
\begin{split}
S_m(x, y)
&\equiv e^{-im y_{2n-1}}\frac{1}{2\pi}\int_{-\pi}^{\pi}\hat S_m(x, (w, \theta))e^{im\theta}d\theta \mod O(m^{-\infty})
\end{split}
\end{equation}
holds on $D\times D$.
Substituting (\ref{b2}) to (\ref{b3}), we have
\begin{equation}\label{i1}
\begin{split}
&S_m(x, y)\equiv \frac{1}{2\pi}e^{im(x_{2n-1}-y_{2n-1}+\Phi(z, w))}\hat b(z, w, m)\mod O(m^{-\infty}),\\
&S_m(x, x)\equiv \frac{1}{2\pi}\hat b(z, z, m)\mod O(m^{-\infty}).
\end{split}
\end{equation}
Thus, from (\ref{i1}) we have

\begin{theorem}\label{c2}
Let $X$ be as in Theorem \ref{t1}. For  $x_0\in X_{\rm reg}$, let $(z, \theta, \varphi)$ be canonical coordinates centered at $x_0$ and defined on a canonical patch $D_1=\{(z, \theta): |z|<\varepsilon_1, |\theta|<\pi\}$. Set  $D=\{(z, \theta)\in\mathbb C^{n-1}\times\mathbb R: |z|<\varepsilon, |\theta|<\frac{\pi}{2}\}\Subset D_1$.
Then on $D\times D$,
we have
\begin{equation*}
\begin{split}
&S_m(x, y)\equiv\frac{1}{2\pi}e^{im(x_{2n-1}-y_{2n-1}+\Phi(z, w))}\hat b(z, w, m)\mod O(m^{-\infty}),
\end{split}
\end{equation*}
where
\begin{equation}\label{Eq: Szego reg 02}
\begin{split}
&\hat b(z, w,m)\sim\sum^\infty_{j=0}m^{n-1-j}\hat b_j(z,w)\ \ \mbox{in $S^{n-1}_{{\rm loc\,}}(1,\tilde D\times\tilde D)$},\\
&\hat b_j(z, w)\in C^{\infty}(\tilde D\times \tilde D),\ \  j=0, 1, 2, \cdots,\\
&\hat b_0(z, z)=\pi^{-(n-1)}\abs{\det\mathcal L_x},\ \ x=(z,0),\ \ \forall z\in\tilde D.\end{split}
\end{equation}
Here, we set $\tilde D=\set{z\in\mathbb C^{n-1}: \abs{z}<\varepsilon}$. In particular, we have
\begin{equation}\label{e-gue151124II}
S_m(x, x)\equiv\frac{1}{2\pi}\hat b(z, z, m)\mod O(m^{-\infty}).
\end{equation}
\end{theorem}

\subsection{Asymptotic Szeg\H{o}  kernel expansion  on the complement of $X_{\rm reg}$ }\label{Sec: asymp exp on comp Xreg}
In this section, we try to get the asymptotic expansion of the Szeg\H{o} kernel on the complement of $X_{\rm reg}$. Fix $x_0\in X_k$ for some $k>1$, where $X_k$ is defined in (\ref{e-gue150614}). Let $(z, \theta, \varphi)$ be canonical coordinates centered at $x_0$ and defined on a canonical chart  $D_1=\{(z, \theta): |z|<\varepsilon_1, |\theta|<\frac{\pi}{k}-\epsilon\}$.
It is straightforward to see that there is a small neighborhood $D=\{(z, \theta): |z|<\varepsilon, |\theta|<\delta\}\Subset D_1$ of $x_0$ such that
\begin{equation}\label{e-gue151123}
e^{i\theta}\circ (0,0)\neq (z,\hat\theta),\ \ \forall\theta\in[0,2\pi),\ \ (z, \hat\theta)\in D,\ \ z\neq0.
\end{equation}
From Theorem \ref{b6}, we have
\begin{equation}\label{b8}
\begin{split}
S_m(x, x_0)
&\equiv\frac{1}{2\pi}\sum_{s=1}^k\int_{\frac{2\pi}{k}(s-1)}^{\frac{2\pi}{k}s}
\hat S_m(x, e^{i\theta}\circ (0, 0))e^{im\theta}d\theta\mod O(m^{-\infty})\\
&\equiv\frac{1}{2\pi}\sum_{s=1}^ke^{i\frac{2\pi}{k}(s-1)m}\int_0^{\frac{2\pi}{k}}\hat S_m(x, e^{i\theta}\circ (0, 0))e^{im\theta}d\theta\mod O(m^{-\infty})
\end{split}
\end{equation}
for any $x\in D$. By a direct calculation, we find
\begin{equation}\label{b7}
\sum_{s=1}^ke^{i\frac{2\pi}{k}(s-1)m}=\left\{
                                        \begin{array}{ll}
                                          k, ~\text{if}~ k\mid m  & \hbox{;} \\
                                          0, ~\text{if}~ k\nmid m& \hbox{.}
                                        \end{array}
                                      \right.
\end{equation}
From \eqref{e-gue151123}, we can check that
\begin{equation}\label{e-gue151123I}
\frac{k}{2\pi}\int_{-\frac{\pi}{k}}^{\frac{\pi}{k}}
\hat S_m(x, e^{i\theta}\circ (0, 0))e^{im\theta}d\theta
=\frac{k}{2\pi}\int_{-\frac{\pi}{k}}^{\frac{\pi}{k}}
\hat S_m(x, (0, \theta))e^{im\theta}d\theta
\end{equation}
holds. Substituting (\ref{b7}) to (\ref{b8}) for $k\mid m$ and using \eqref{e-gue151123I}, we have
\begin{equation}\label{b9}
\begin{split}
S_m(x, x_0)
&\equiv\frac{k}{2\pi}\int_{-\frac{\pi}{k}}^{\frac{\pi}{k}}
\hat S_m(x, (0, \theta))e^{im\theta}d\theta\mod O(m^{-\infty}).
\end{split}
\end{equation}
Substituting (\ref{b2}) to (\ref{b9}) yields
\begin{equation*}
\begin{split}
S_m(x, x_0)
&\equiv \frac{k}{2\pi}e^{im(x_{2n-1}+\Phi(z, 0))}\hat b(z, 0, m)\mod O(m^{-\infty}).
\end{split}
\end{equation*}
Summing up, we obtain

\begin{theorem}\label{e2}
Let $X$ be as in Theorem \ref{t1}. Assume $x_0\in X_k, k>1$. Let $D_1$, $D$ and $(z, \theta, \varphi)$ be as above with $\frac{\pi}{k}-\epsilon$ replaced by $\delta_1$.
For $k\nmid m$ we have $S_m(x, x_0)=0$ for all $x\in D.$
For $k\mid m$ we have
\begin{equation}\label{e3}
S_m(x, x_0)\equiv \frac{k}{2\pi}e^{im(x_{2n-1}+\Phi(z, 0))}\hat b(z, 0, m)\mod O(m^{-\infty})
\end{equation}
on $D$. In particular, given $k\mid m$ and $x=x_0$, we have
\begin{equation*}
S_m(x_0, x_0)=\frac{k}{2\pi}\hat b(0, 0, m)+O(m^{-\infty})
\end{equation*}
and
\begin{equation*}
\hat b(0, 0, m)\sim \hat b_0(0, 0)m^{n-1}+\hat b_1(0, 0)m^{n-2}+\cdots
\end{equation*}
in the sense that for any $N\in\mathbb N_{0}$ there exists $C_N>0$ independent of $m$ such that
\begin{equation*}
\left|\hat b(0, 0, m)-\sum_{j=0}^N\hat b_j(0, 0)m^{n-1-j}\right|\leq C_Nm^{n-2-N}
\end{equation*}
holds for all $m\in\N$.
\end{theorem}

\section{Equivariant embedding of CR manifolds }\label{sec3}
Let $X$ be a compact strongly pseudoconvex CR manifolds with a locally free transversal CR $S^1$-action. Now we use the Szeg\H{o}  kernel expansion we have established in Section \ref{sec1} to get the equivariant embedding of $X$.

\subsection{Immersion of CR manifold }
We assume $X=X_1\cup X_2\cup\cdots\cup X_l, ~X_1\neq \emptyset,$ where $X_k$, $1\leq k\leq l$, is defined in (\ref{e-gue150614}). For $1\leq k\leq l$ let $\{f_j^k\}_{j=1}^{d_{km}}$ and $\{g_j^k\}_{j=1}^{d_{k(m+1)}}$ be orthonormal bases of $H^0_{b, km}(X)$ and $H^0_{b, k(m+1)}(X)$, respectively. Now for $1\leq k\leq l$ we can define a CR map from $X$ to Euclidean space as follows
\begin{equation*}
\Phi^k_m: X\rightarrow \mathbb C^{d_{km}+d_{k(m+1)}}, x\mapsto (f_1^k(x),\cdots, f_{d_{km}}^k(x), g_1^k(x), \cdots, g_{d_{k(m+1)}}^k(x) ).
\end{equation*}
Combining the $\Phi_{m}^ks, 1\leq k\leq l$, we define a CR map
\begin{equation*}
\Phi_m: X\rightarrow \mathbb C^{N_m}, x\rightarrow (\Phi^1_m(x), \cdots, \Phi^l_m(x)),
\end{equation*}
where $N_m=\sum\limits_{k=1}^l (d_{km}+d_{k(m+1)}).$ If the transversal CR $S^1$-action on $X$ is globally free, then $X=X_1=X_{\rm reg}$ and Epstein \cite{Ep92} showed that $\Phi_m^1$ is an CR embedding when $m$ is large. However, if the transversal CR $S^1$-action on $X$ is just locally free the CR functions in $H^0_{b, m}(X)\bigoplus H^0_{b, m+1}(X)$ are not enough for the embedding. The reason is that the space $H^0_{b, m}(X)\bigoplus H^0_{b, m+1}(X)$ will be not enough to separate the points in $X\setminus X_{\rm reg}$.

Now we use the asymptotic Szeg\H{o}  kernel expansion in Section \ref{sec1} to establish

\begin{lemma}\label{k1}
The map $\Phi_m: X\rightarrow \mathbb C^{N_m}$ is an immersion when $m$ is large.
\end{lemma}

\begin{proof}
For $x_0\in X_k$, let $(z, \theta, \varphi)$ be canonical coordinates centered at $x_0$ and defined on a canonical local patch $D=\{(z, \theta): |z|<\varepsilon, |\theta|<\delta\}=\tilde D\times(-\delta,\delta)$. Assume that $k|m$. Let $\{f_j\}_{j=1}^{d_m}\subset H^0_{b, m}(X)$ be an orthonormal basis.
Since $S_m(x, y)=\sum\limits_{j=1}^{d_m}f_j(x)\overline{f_j(y)}$, we have that $\overline{S_m(x, y)}=S_m(y, x)$. Furthermore, for any $u\in C_0^\infty(D)$ we have $S_m u(x)=\int_{D}S_m(x, y)u(y)dv_X(y)$ and hence
\begin{equation}\label{c1}
\overline{S_m u}=\int_{D}\overline{S_m(x, y)}\overline {u(y)}dv=\int_{D}S_m(y, x)\overline{u(y)}dv_X.
\end{equation}
Choose cut-off functions $\chi\in C_0^\infty(\mathbb C^{n-1}), \chi_2\in C_0^\infty(-\delta, \delta)$ such that ${\rm supp}\chi\Subset\{w\in\mathbb C^{n-1}: |w|<1\}$ and $\int_{-\delta}^{\delta}\chi_2(y_{n-1})dy_{2n-1}=1$.
For $j=1, \cdots, n-1$, set
\begin{equation}\label{c6}
u_j(y)=w_j\chi\left(\frac{\sqrt m w}{\log m}\right)\chi_2(y_{2n-1})e^{im y_{2n-1}}e^{im {\rm Re}\Phi(w, 0)},\end{equation}
where $\Phi$ is as in Theorem~\ref{b6}.
Then $u_j\in C_0^\infty(X)$ has its support in $D$ if $m$ is sufficiently large.
Define $v_j=S_m u_j, j=1, \cdots, n-1.$
 Then from Theorem \ref{b6} and (\ref{c1}) we have
\begin{equation*}
\begin{split}
\overline{S_m u_j(x)}
&=\frac{1}{2\pi}\int_{D}\int_{-\pi}^{\pi}\hat S_m(y, e^{i\theta}\circ x)e^{im\theta}d\theta\overline{u_j(y)}dv_X+\int_D \overline{R_m(x, y)}\overline {u_j(y)}dv_X,
\end{split}
\end{equation*}
where $R_m(x,y)\equiv0\mod O(m^{-\infty})$.
With respect to the canonical local coordinates, one notes that
\begin{equation}\label{trick1}
\frac{\partial \hat S_m(y, e^{i\theta}\circ x)}{\partial \overline z_j}\Big|_{x=x_0}=\frac{\partial \hat S_m}{\partial \overline z_j}(y, e^{i\theta}\circ x_0).
\end{equation}
Then by (\ref{trick1}) and a direct calculation we have
\begin{equation}\label{c4}
\begin{split}
\frac{\partial\overline{S_m u_j}}{\partial\overline {z_j}}( x_0)
&\equiv\frac{k}{2\pi}\int_D\int_{-\frac{\pi}{k}}^{\frac{\pi}{k}}
\frac{\partial \hat S_m}{\partial\overline{z_j}}(y, (0, \theta))
e^{im\theta}d\theta\overline{u_j(y)}dv_X+O(m^{-\infty})
\end{split}
\end{equation}
which leads to
\begin{equation}\label{c3}
\begin{split}
&\frac{\partial\hat S_m}{\partial{\overline z_j}}(y, (0, \theta))=e^{im(y_{2n-1}-\theta+\Phi(w, 0))}\eta(\theta)\times\\
&\left[im\frac{\partial\Phi(w, 0)}{\partial\overline z_j}\hat b(w, 0, m)\tilde\chi_1(0)+\frac{\partial \hat b(w, 0, m)}{\partial\overline z_j}\tilde\chi_1(0)+\hat b(w, 0, m)\frac{\partial\tilde\chi_1}{\partial \overline z_j}(0)\right]\\
&=e^{im(y_{2n-1}-\theta+\Phi(w, 0))}\eta(\theta)\left[2m(\lambda_jw_j+O(|w|^2))\hat b(w, 0, m)+\frac{\partial\hat b(w, 0, m)}{\partial\overline z_j}\right].
\end{split}
\end{equation}
Substituting (\ref{c3}) to (\ref{c4}), we have
\begin{equation}\label{c5}
\begin{split}
\frac{\partial\overline{S_m u_j}}{\partial\overline {z_j}}( x_0)
=&\frac{k}{2\pi}\int_D
e^{im(y_{2n-1}+\Phi(w, 0))}
\left[2m(\lambda_jw_j+O(|w|^2))\hat b(w, 0, m)+\frac{\partial\hat b(w, 0, m)}{\partial\overline z_j}\right]\times\\
&\overline{u_j(y)}dv_X+O(m^{-\infty}).
\end{split}
\end{equation}
Substituting (\ref{c6}) to (\ref{c5}) and taking the coordinate transformation $w\rightarrow \frac{w}{\sqrt m}$, we have
\begin{equation}\label{l7}
\begin{split}
\frac{\partial\overline{S_m u_j}}{\partial\overline {z_j}}( x_0)
&=\frac{k}{2\pi}\int_{\abs{w}\leq\log m}e^{-m {\rm Im}\Phi(\frac{w}{\sqrt m}, 0)}m^{-(n-1)}\times\\
&\left[2(\lambda_j|w_j|^2+\frac{1}{\sqrt m}O(|w|^3))\hat b(\frac{w}{\sqrt m}, 0, m)+\frac{1}{\sqrt m}\frac{\partial\hat b(\frac{w}{\sqrt m}, 0, m)}{\partial\overline z_j}\overline w_j\right]\times\\
&\chi(\frac{w}{\log m})\lambda(\frac{w}{\sqrt m})dv(w)+O(m^{-\infty}),
\end{split}
\end{equation}
and hence
\begin{equation}\label{d2}
\lim_{m\rightarrow\infty}\frac{\partial\overline{S_m u_j}}{\partial\overline {z_j}}( x_0)=\frac{k}{2\pi}\int_{\mathbb C^{n-1}}e^{-\lambda|w|^2}2\lambda_j|w_j|^2\hat b_0(0,0)dv(w)=c_j\neq 0,
\end{equation}
where
$\lambda |w|^2=\sum_{j=1}^{n-1}\lambda_j|w_j|^2$ and $c_j$ is a non-zero real number.

For $j\neq k$, we can repeat the procedure above and get
\begin{equation}\label{l8}
\begin{split}
&\frac{\partial\overline{S_m u_j}}{\partial\overline z_k}(x_0)=\frac{k}{2\pi}\int_{\abs{w}\leq\log m}e^{-m {\rm Im}\Phi(\frac{w}{\sqrt m}, 0)}m^{-(n-1)}\times\\
&\left[(2\lambda_k w_k\overline w_j+\frac{1}{\sqrt m}O(|w|^3))\hat b(\frac{w}{\sqrt m}, 0, m)+\frac{1}{\sqrt m}\frac{\partial\hat b(\frac{w}{\sqrt m}, 0, m)}{\partial\overline z_k}\overline w_j\right]\times\\
&\chi(\frac{w}{\log m})\lambda(\frac{w}{\sqrt m})dw+O(m^{-\infty}).
\end{split}
\end{equation}
Letting $m\rightarrow\infty$, we get
\begin{equation}\label{d3}
\lim_{m\rightarrow\infty}\frac{\partial\overline{S_m u_j}}{\partial\overline z_k}(x_0)=\frac{k}{2\pi}\int_{\mathbb C^{n-1}}e^{-\lambda|w|^2}2\lambda_kw_k\overline w_j\hat b_0(0,0)dv(w)=0.
\end{equation}

Similarly, one computes
\begin{equation}\label{l6}
\begin{split}
&\frac{\partial\overline{S_m u_j}}{\partial z_k}(x_0)=\frac{k}{2\pi}\int_{\abs{w}\leq\log m}e^{-m \rm Im\Phi(\frac{w}{\sqrt m}, 0)}m^{-(n-1)}\times\\
&\left[(2\lambda_k \overline w_k\overline w_j+\frac{1}{\sqrt m}O(|w|^3))\hat b(\frac{w}{\sqrt m}, 0, m)+\frac{1}{\sqrt m}\frac{\partial\hat b(\frac{w}{\sqrt m}, 0, m)}{\partial z_k}\overline w_j\right]\times\\
&\chi(\frac{w}{\log m})\lambda(\frac{w}{\sqrt m})dw+O(m^{-\infty}).
\end{split}
\end{equation}
Letting $m\rightarrow\infty$, we have
\begin{equation}\label{l7}
\lim_{m\rightarrow\infty}\frac{\partial\overline{S_mu_j}}{\partial z_k}(x_0)=\frac{k}{2\pi}\int_{\mathbb C^{n-1}}e^{-\lambda|w|^2}2\lambda_k \overline w_k\overline w_j\hat b_0(0,0)dv(w)=0.
\end{equation}


Given $j=n$, choose $\chi_3(y_{2n-1})\in C_0^\infty(-\delta_1, \delta_1)$ satisfying
$
\int_{-\delta_1}^{\delta_1}y_{2n-1}\chi_3(y_{2n-1})=1.
$
Set
\begin{equation*}
u_n=m y_{2n-1}\chi_3(my_{2n-1})e^{im y_{2n-1}}\chi\left(\frac{\sqrt mw}{\log m}\right)e^{im\rm Re\Phi(w, 0)},\ \ v_n=S_mu_n.
\end{equation*}
Then by (\ref{trick1}) and the same argument as in (\ref{b8}) we have
\begin{equation}\label{c8}
\begin{split}
\frac{\partial\overline{S_m u_n}}{\partial x_{2n-1}}(x_0)
&=\frac{k}{2\pi}\int_D\int_{-\frac{\pi}{k}}^{\frac{\pi}{k}}\frac{\partial\hat S_m}{\partial x_{2n-1}}(y, e^{i\theta}\circ x_0)e^{im\theta}d\theta\overline{u_n(y)}dv_X
+O(m^{-\infty}).
\end{split}
\end{equation}
By a direct calculation, we have
\begin{equation}\label{c7}
\begin{split}
\frac{\partial\hat S_m}{\partial x_{2n-1}}(y, 0, \theta)=e^{im(y_{2n-1}-\theta+\Phi(w, 0))}\hat b(w, 0, m)\left[-im\eta(\theta)+
\frac{\partial\eta(\theta)}{\partial\theta}\right].
\end{split}
\end{equation}
Substituting (\ref{c7}) to  (\ref{c8}) and using the fact that $\int_{-\frac{\pi}{k}}^{\frac{\pi}{k}}\frac{\partial\eta(\theta)}{\partial\theta}d\theta=0$, we find
\begin{equation}\label{c9}
\begin{split}
&\frac{k}{2\pi}\int_D\int_{-\frac{\pi}{k}}^{\frac{\pi}{k}}\frac{\partial\hat S_m}{\partial x_{2n-1}}(y, e^{i\theta}\circ x_0)e^{im\theta}d\theta\overline{u_n(y)}dv_X\\
=&\frac{-ik}{2\pi}\int_{\abs{w}\leq\log m}m^{-(n-1)}\hat b(\frac{w}{\sqrt m}, 0, m)e^{-m{\rm Im}\Phi(\frac{w}{\sqrt m}, 0)}\lambda(w)dv(w).
\end{split}
\end{equation}
Substituting (\ref{c9}) to (\ref{c8}) and letting $m\rightarrow\infty$, we have

\begin{equation}\label{d1}
\lim_{m\rightarrow\infty}\frac{\partial\overline{S_m u_n}(x_0)}{\partial x_{2n-1}}=\frac{-ik}{2\pi}\hat b_0(0, 0)\int_{\mathbb C^{n-1}}e^{-\lambda |w|^2}dv(w)=ic_n\neq 0,
\end{equation}
where $c_n$ is a nonzero real number.

On the other hand, for $j=1, \cdots, n-1$ by a similar calculation  we have
\begin{equation}\label{l3}
\begin{split}
\frac{\partial\overline{S_m u_n}}{\partial \overline z_j}(x_0)
=&\frac{k}{2\pi}\int_{\abs{w}\leq\log m}e^{-m{\rm Im}\Phi(\frac{w}{\sqrt m}, 0)}
[2(\lambda_j\frac{w_j}{\sqrt m}+\frac{1}{m}O(|w|^2))\hat b(\frac{w}{\sqrt m}, 0, m)\\
&+\frac{1}{m}\frac{\partial\hat b}{\partial\overline z_j}(\frac{w}{\sqrt m}, 0, m)]
\chi(\frac{w}{\log m})\lambda(\frac{w}{\sqrt m})m^{-(n-1)}dv(w).
\end{split}
\end{equation}
By (\ref{l3}) we get
\begin{equation}\label{l4}
\left|\frac{\partial\overline{S_m u_n}}{\partial \overline z_j}(x_0)\right|\leq C\frac{1}{\sqrt m},
\end{equation}
where $C$ is a constant which does not depend on $x_0$ and $m$.
Similarly, we find
\begin{equation}\label{l5}
\left|\frac{\partial\overline{S_m u_n}}{\partial  z_j}(x_0)\right|\leq C\frac{1}{\sqrt m}.
\end{equation}
Set $v_j=\alpha_{2j-1}+i\alpha_{2j}$, $j=1, \cdots, n$. Then combining the above arguments there are positive constants $c$, $C$ independent of $x_0$ and $m$ and a sequence $\{\varepsilon_m\}$  which does not depend on $x_0\in X$ with $\varepsilon_m\rightarrow 0$ as $m\rightarrow\infty$ such that the following estimates hold
\begin{equation}\label{l9}
\begin{split}
&\left|\frac{\partial \alpha_j}{\partial x_j}(x_0)\right|\geq c; \left|\frac{\partial \alpha_{2n}}{\partial x_{2n-1}}(x_0)\right|\geq c , j=1, \cdots, 2n-2,\\
&\left|\frac{\partial\alpha_j}{\partial x_k}(x_0)\right|\leq \varepsilon_m, j\neq k,  j, k=1, \cdots 2n-2,\\
&\left|\frac{\partial\alpha_{2n}}{\partial x_j}(x_0)\right|\leq C\frac{1}{\sqrt m}, j=1, \cdots, 2n-2.
\end{split}
\end{equation}
From (\ref{l9}) the real Jacobian matrix of $\Phi_m$ is non-degenerate at any $x_0\in X$ when $m$ is large enough which
implies that $\Phi_m$ is an immersion. Thus, we get the conclusion of the lemma.
\end{proof}

\subsection{Analysis near the complement of $X_{\rm reg}$}\label{s-f}
In order to get the global embedding of CR manifolds by CR functions which lie in the positive Fourier components we need the following
\begin{proposition}\label{p1}
Fix $x_0\in X\setminus X_{\rm reg}$. Without loss of generality, we assume $x_0\in X_{k_0}$ for some $k_0>1$. We have
\begin{enumerate}
  \item There exist a positive integer $m_0$ and a neighborhood $U(x_0)$ of $x_0$ such that $\Phi^{k_0}_{m_0}: U(x_0)\rightarrow \mathbb C^{d_{k_0m_0}+d_{k_0(m_0+1)}}$ is an embedding  and $S_{k_0m_0}(x, x_0)\neq 0$, $S_{k_0(m_0+1)}(x,x_0)\neq0$, for all $ x\in U(x_0)$.
\item There exist positive constants $\varepsilon_0, \delta_0$ and a neighborhood $V(x_0)$ of $x_0$ with $V(x_0)\Subset U(x_0)$ such that
\begin{equation*}
	\begin{split}
		&e^{i\theta}\circ V(x_0)\subset U(x_0), \forall \theta\in I(x_0, \varepsilon_0),\\
		&-1\leq\cos k_0\theta\leq 1-\delta_0, \forall \theta\not\in I(x_0, \varepsilon_0), 0\leq\theta<2\pi
	\end{split}
\end{equation*}
holds, where $I(x_0, \varepsilon_0)$ is given by
\begin{equation}\label{e6}
\begin{split}
I(x_0, \varepsilon_0)
&=\{\theta: 0\leq\theta<\varepsilon_0\}\cup\{\theta: |\theta-\frac{2\pi}{k_0}|<\varepsilon_0\}\cup\{\theta: |\theta-\frac{4\pi}{k_0}|<\varepsilon_0\}\cup\cdots\\
&\quad\cup\{\theta: |\theta-\frac{2(k_0-1)\pi}{k_0}|<\varepsilon_0\}\cup
\{\theta: 2\pi-\varepsilon_0<\theta<2\pi\}.
\end{split}
\end{equation}
\item Fix $0<\sigma<\frac{\delta_0}{100}$, where $\delta_0>0$ is as in (2). There exist a positive integer $m_1$ and a neighborhood $W(x_0)$ of $x_0$ with $W(x_0)\Subset V(x_0)$ such that $S_{k_0m_1}(x,x_0)\neq0$ for all $x\in W(x_0)$ and the real part of $\frac{S_{k_0(m_1+1)}(x, x_0)}{S_{k_0m_1}(x, x_0)}$ denoted by $\mathcal R_{k_0m_1}(x)$ satisfies
\begin{equation*}|1-\mathcal R_{k_0m_1}(x)|<\sigma, \forall x\in W(x_0).\end{equation*}
The imaginary part of $\frac{S_{k_0(m_1+1)}(x, x_0)}{S_{k_0m_1}(x, x_0)}$ denoted by $\mathcal I_{k_0m_1}(x)$ satisfies the following inequality
\begin{equation*}
|\mathcal I_{k_0m_1}(x)|<\frac{\sigma}{8}, \forall x\in W(x_0).
\end{equation*}
\item For any positive constant $c>0$, there exist a positive integer $m_2$ and a neighborhood $\hat W(x_0)\Subset W(x_0)$ of $x_0$ such that
\begin{equation*}
|S_{k_0m_2}(x, x_0)|>\frac{c}{2}, \forall  x\in\hat W(x_0)
\end{equation*}
and
\begin{equation*}
|S_{k_0m_2}(y, x_0)|<\frac{c}{8}, \forall y \not\in\bigcup\limits_{0\leq\theta<2\pi}e^{i\theta}\circ W(x_0)
\end{equation*}
hold.
\end{enumerate}
\end{proposition}

\begin{proof}
Fix $x_0\in X_{k_0}$ and let $D$ be the canonical local patch given in Theorem \ref{j}. From (\ref{e3}), we have for any $D^\prime\Subset D$ and $N\in\mathbb N$, there exists a constant $C_{D^\prime , N}$ such that
\begin{equation}\label{e5}
|S_{k_0m}(x, x_0)|\geq \frac{k_0}{2\pi}\abs{\hat b(z, 0, k_0m)}e^{-k_0m{\rm Im}\Phi(z, 0)}-C_{D^\prime, N} m^{-N}, m>>1
\end{equation}
holds. Given $x=(z, \theta)$ with $|z|\leq \frac{1}{m}, |\theta|\leq \frac{1}{m}$, one has $|S_{k_0m}(x, x_0)|>0$ when $m\gg1$. Thus, there is a $\lambda_0>0$ such that for all $m\geq \lambda_0$ we have $|S_{k_0m}(x, x_0)|>0$ for all $x\in U_{m}(x_0)$, where $U_{m}(x_0)=\{(z, \theta): |z|<\frac{1}{m}, |\theta|<\frac{1}{m}\}$. Moreover, from the proof of Lemma~\ref{k1}, we see that there is a $\lambda_1>0$ such that for all $m\geq \lambda_1$, there is a small neighborhood $\Td U_m(x_0)$ of $x_0$ such that $\Phi^{k_0}_{m}: \Td U_m(x_0)\rightarrow \mathbb C^{d_{k_0m_0}+d_{k_0(m_0+1)}}$ is an embedding. Taking $m_0\geq \lambda_0+\lambda_1$ and setting  $U(x_0)=U_{m_0}(x_0)\bigcap U_{m_0+1}(x_0)\bigcap\Td U_{m_0}(x_0)$, we get (1).

Since $x_0\in X_{k_0}$,  we have $e^{i\frac{2\pi}{k_0}j}\circ x_0=x_0$ for $0\leq j\leq k_0, j\in \mathbb Z.$ Then for any $\varepsilon_0$ we define $I(x_0, \varepsilon_0)$ as in (\ref{e6}). When $\varepsilon_0$ is sufficiently small there exists a small neighborhood of $x_0$ denoted by $V(x_0)\Subset U(x_0)$ such that  $e^{i\theta}\circ V(x_0)\subset U(x_0)$ for $\theta\in I(x_0, \varepsilon_0)$. For $\theta\notin I(x_0, \varepsilon_0)$, $0\leq\theta<2\pi$, we have
$|k_0\theta-2\pi j|\geq\varepsilon_0k_0$ for every $j=0,1,\ldots,k_0$ which implies that there exists a constant $\delta_0$ depending on $\varepsilon_0$ such that $-1\leq\cos k_0\theta\leq 1-\delta_0$ for $\theta\not\in I(x_0, \varepsilon_0)$. Thus we get the conclusion of (2) in this proposition.

From the proof of (1), there is an $\Td m_1>0$ such that for every $m\geq\Td m_1$, there is a neighborhood $W_m(x_0)$ of $x_0$ such that $S_{k_0m}(x, x_0)\neq0$ and $S_{k_0(m+1)}(x, x_0)\neq0$. We assume that $m\geq\Td m_1$ and $x\in W_m(x_0)$.
By (\ref{e3}), we have
\begin{equation}\label{j5}
\begin{split}
&S_{k_0m}(x, x_0)\equiv \frac{k_0}{2\pi}e^{ik_0m(x_{2n-1}+\Phi(z, 0))}\hat b(z, 0, m)\mod O(m^{-\infty}),\\
&S_{k_0(m+1)}(x, x_0)\equiv \frac{k_0}{2\pi}e^{ik_0(m+1)(x_{2n-1}+\Phi(z, 0))}\hat b(z, 0, m+1)\mod O(m^{-\infty}),\\
&\hat b(z, 0, m)\sim \sum_{j=0}^\infty \hat b_{j}(z, 0)m^{n-1-j}~\text{in}~S^{n-1}_{\rm loc}(1; D).
\end{split}
\end{equation}
Write
\begin{equation*}
\frac{S_{k_0(m+1)}(x, x_0)}{S_{k_0m}(x, x_0)}=\mathcal R_{k_0m}(x)+i\mathcal I_{k_0m}(x).
\end{equation*}
Since $\hat b_0(0,0)\neq0$ (see Theorem~\ref{c2}), we have $\hat b(0, 0,m)\neq 0$ for $m$ large and this implies that $\hat b(z, 0,m)\neq 0$ when $|z|$ is sufficiently small. We assume that  $\hat b(z, 0,m)\neq 0$ for every $m\geq\Td m_1$ and every $(z,0)\in W_m(x_0)$. Set
\begin{equation*}
\begin{split}
a_m(x)=\frac{k_0}{2\pi}e^{ik_0m(x_{2n-1}+\Phi(z, 0))}\hat b(z, 0, m), b_m(x)=S_{k_0m}(x, x_0)-a_m(x).
\end{split}
\end{equation*}
From (\ref{j5}), for any $D^\prime\Subset V(x_0)\Subset D$ and any $N\in\mathbb N$ there exists a positive constant $C_{D^\prime, N}$ such that
\begin{equation*}
\sup_{x\in D^\prime}|S_{k_0m}(x)-a_m(x)|\leq C_{D^\prime, N}m^{-N},  m>>1,
\end{equation*}
holds. For any $m\geq\Td m_1$, define $V_m(x_0)=\{x=(z, \theta)\in D, |z|<\frac{1}{m}, |\theta|<\frac{1}{m}\}\bigcap W_m(x_0)$, then $V_m(x_0)\Subset D^\prime$ when $m$ is sufficiently large. Then on $V_m(x_0)$, we have
\begin{equation}\label{j8}
|b_{m+1}(x)|\leq C_{D^\prime, N}\frac{1}{(m+1)^N}, |b_m(x)|\leq C_{D^\prime, N}\frac{1}{m^N}.
\end{equation}
On the other hand, we have $|a_m(x)|=\frac{k_0}{2\pi}e^{-k_0m{\rm Im}\Phi(z, 0)}\hat b(z, 0, m)$. From (\ref{j6}), by a direct calculation we get
$
{\rm Im}\Phi(z, 0)=\lambda|z|^2+O(|z|^3).
$
So we assume $D^\prime$ to be sufficiently small such that on $D^\prime$ we have
\begin{equation*}
c_1|z|^2\leq {\rm Im}\Phi(z, 0)\leq c_2|z|^2
\end{equation*}
for some constants $c_1, c_2$. Then
\begin{equation}\label{j7}
|a_m(x)|\geq \hat c m^{n-1}, \forall x\in V_{m}(x_0),\ \ \frac{a_{m+1}(x)}{a_m(x)}\approx 1,\ \ \forall x\in V_m(x_0),
\end{equation}
holds for some positive constant $\hat c$ when $m$ is sufficiently large.
Since
\begin{equation*}
\frac{S_{k_0(m+1)}(x, x_0)}{S_{k_0m}(x, x_0)}=\frac{b_{m+1}+a_{m+1}}{b_m+a_m}=\frac{\frac{b_{m+1}}{a_m}+
\frac{a_{m+1}}{a_m}}{\frac{b_m}{a_m}+1},
\end{equation*}
(\ref{j8}) and (\ref{j7}) imply
\begin{equation*}
\frac{S_{k_0(m+1)}(x, x_0)}{S_{k_0m}(x, x_0)}\approx1, \forall x\in V_m(x_0)
\end{equation*}
for $m>>1$.
Then for any fixed $0<\sigma<\frac{\delta_0}{100}$, we can
choose $m_1$ sufficiently large such that $W(x_0)=\{(z, \theta): |z|<\frac{1}{m_1}, |\theta|<\frac{1}{m_1}\}$ satisfies $W(x_0)\Subset V(x_0)$ and on $W(x_0)$ we have
\begin{equation}
|1-\mathcal R_{k_0m_1}(x)|<\sigma, |\mathcal I_{k_0m_1}(x)|<\frac{\sigma}{8}.
\end{equation}
Thus, we get the conclusion of (3) in the proposition.

Choose a neighborhood $W_1(x_0)$ of $x_0$ such that $W_1(x_0)\Subset W(x_0)$ holds. Following the same arguments as in the proof of Lemma \ref{f5}, we have
\begin{equation}\label{j2}
S_{k_0m}(x_0, y)\equiv 0\mod O(m^{-\infty}), \forall y\not\in \bigcup\limits_{0\leq \theta<2\pi}e^{i\theta}\circ\overline{W_1(x_0)}.
\end{equation}
Since $X\setminus \bigcup\limits_{0\leq \theta<2\pi}e^{i\theta}\circ{W(x_0)}\Subset X\setminus \bigcup\limits_{0\leq \theta<2\pi}e^{i\theta}\circ\overline{W_1(x_0)}$, (\ref{j2}) implies that for any $N>0$ there exists a constant $C_N$ satisfying
\begin{equation*}
|S_{k_0 m}(x_0, y)|\leq C_N m^{-N}~\forall m>>1, \forall y\in X\setminus \bigcup\limits_{0\leq \theta<2\pi}e^{i\theta}\circ{W(x_0)} .
\end{equation*}
Thus, for any $c>0$, there exists $n_0$ such that for any $m>n_0$ we have $|S_{k_0 m}(x_0, y)|<\frac{c}{8}$ for all $y\not\in \bigcup\limits_{0\leq \theta<2\pi}e^{i\theta}\circ W(x_0)$. Then following the same arguments as in the proof of (1) in the proposition, there exists a positive integer $m_2$ and a neighborhood $\hat W(x_0)\Subset W_1(x_0)\Subset W(x_0)$ such that $|S_{k_0m_2}(x, x_0)|>\frac{c}{2}$ holds for all $x\in \hat W (x_0)$ and moreover $|S_{k_0 m_2}(x_0, y)|<\frac{c}{8}$ holds for all $y\not\in X\setminus \bigcup\limits_{0\leq \theta<2\pi}e^{i\theta}\circ{W(x_0)}$. Thus, we get the conclusion of (4) in this proposition.
\end{proof}
\subsection{Embedding of CR manifold by positive Fourier components}
Now, we are going to establish the global embedding of the CR manifolds, which have a locally free transversal CR $S^1$-action, by positive Fourier components.

Since $X\setminus X_{\rm reg}\Subset X$, there exist finite $\hat W(x_i)\Subset W(x_i)\Subset V(x_i)\Subset U(x_i)$ and positive constants $m_0(x_i), m_1(x_i), m_2(x_i)$ with respect to the points $x_i$, $ 0\leq i\leq n_0$ satisfying the properties in Proposition \ref{p1} and moreover $X\setminus X_{\rm reg}=\cup_{i=1}^{n_0}\hat W(x_i)$. Without loss of generality,  we assume that $x_i\in X_{k_i}, 0\leq i\leq n_0$. For every $i=0,1,\ldots,n_0$, set
\begin{equation*}
\begin{split}
&H_{x_i}=\bigoplus_{j=0}^2\Bigr(H^0_{b, k_im_j(x_i)}(X)\bigoplus H^0_{b, k_i(m_j(x_i)+1)}(X)\Bigr),\\
&H_m=\bigoplus_{k=1}^l\Bigr(H^0_{b, km}(X)\bigoplus H^0_{b, k(m+1)}(X)\Bigr)\bigoplus_{i=0}^{n_0}H_{x_i}.
\end{split}
\end{equation*}
Recall that $X=X_1\cup X_2\cup\cdots\cup X_l$, $X_{\rm reg}=X_1\neq\emptyset$. Put $N_m={\rm dim} H_m$ and let $\{f_j\}_{j=1}^{N_m}$ be an orthonormal basis of $H_m$ with respect to its decomposition. Define a map
\begin{equation*}
	\Phi_m: X\rightarrow\mathbb C^{N_m}, x\mapsto (f_1(x), \cdots, f_{N_m}(x)).
\end{equation*}
We will prove the following
\begin{theorem}\label{main1}
Let $X$ be a compact connected strongly pseudoconvex CR manifold with a locally free transversal CR $S^1$-action. Then $\Phi_m$ is an embedding when $m$ is large.
\end{theorem}
\begin{proof}
By Lemma \ref{k1}, we know that $\Phi_m$ is an immersion when $m$ is large. Now we show that $\Phi_m$ is injective when $m$ is large by seeking a contradiction. We assume that there exist two sequences $\{\hat y_m\}, \{\hat z_m\}\subset X$, $\hat y_m\neq\hat z_m$ such that $\Phi_m(\hat y_m)=\Phi_m(\hat z_m)$. Since $X$ is compact, there exist  subsequences of $\{\hat y_m\}, \{\hat z_m\}$ which are also denoted by $\{\hat y_m\}, \{\hat z_m\}$ such that $\hat y_m\rightarrow\hat y$, $\hat z_m\rightarrow\hat z$ for $m\rightarrow\infty$.

First we assume that $\hat y, \hat z\in X\setminus X_{\rm reg}.$

Case I: $\hat y=e^{i\theta_0}\circ\hat z, \hat z\in X_k$ for some $k>1$ and $\hat z\in U(x_i)$ for some $i$. By assumption of $\hat y_m, \hat z_m$ we have that
\begin{equation}\label{k2}
\begin{split}
S_{k_im_0(x_i)}(\hat y, x_i)&=S_{k_im_0(x_i)}(\hat z, x_i),\\
S_{k_i(m_0(x_i)+1)}(\hat y, x_i)&=S_{k_i(m_0(x_i)+1)}(\hat z, x_i).
\end{split}
\end{equation}
In the following context, we will omit $x_i$ in $m_j(x_i), j=0, 1, 2$ for brevity if it makes no confusion. Then
(\ref{k2}) implies
\begin{equation*}
\begin{split}
e^{ik_im_0\theta_0}S_{k_im_0}(\hat z, x_i)&=S_{k_im_0}(\hat z, x_i),\\
e^{ik_0(m_0+1)\theta_0}S_{k_i(m_0+1)}(\hat z, x_i)&=S_{k_i(m_0+1)}(\hat z, x_i).
\end{split}
\end{equation*}
By (1) in Proposition \ref{p1}, we have $e^{ik_i\theta_0}=1.$ Then $\theta_0=\frac{2\pi}{k_i}m$ holds for some $m\in\mathbb Z.$
The rigid Hermitian metric on $X$ implies that $e^{i\theta}: X\rightarrow X$ is an isometric map for each $\theta$. Thus, we have
\begin{equation}\label{d4}
{\rm dist}(\hat y, x_i)={\rm dist} (e^{i\frac{2\pi}{k_i}m}\circ\hat z, x_i)={\rm dist}(e^{i\frac{2\pi}{k_i}m}\circ\hat z, e^{i\frac{2\pi}{k_i}m}\circ x_i)={\rm dist}(\hat z, x_i).
\end{equation}
This implies $\hat y\in U(x_i)$ if the $U(x_i)$ we chose is a geodesic ball centered at $x_i$. This is a contradiction since $\Phi_m$ is an embedding on $U(x_i)$.

case II: $\hat y\neq e^{i\theta}\circ\hat z, \forall 0< \theta<2\pi$. We assume $\hat z\in \hat W(x_i)$. Since $\Phi_m$ is an embedding on $U(x_i)$, we must have $\hat y\not\in U(x_i).$
Now we have a claim as follows

Claim: $\hat y\not\in\bigcup\limits_{0\leq\theta<2\pi} e^{i\theta}\circ W(x_i)$.

We prove the claim by seeking a contradiction. If it is not true, there exists a $\hat z_1\in W(x_i)$ such that $\hat y=e^{i\hat \theta}\circ\hat z_1$ holds for some $\hat \theta\in [0, 2\pi)$. By (2) in Proposition \ref{p1} we have $\hat \theta\not\in I(x_i, \varepsilon_i)$ and $-1\leq\cos k_i\hat\theta\leq 1-\delta_i$. Since
\begin{equation*}
\begin{split}
S_{k_im_1}(\hat y, x_i)=S_{k_im_1}(\hat z, x_i),~
S_{k_i(m_1+1)}(\hat y, x_i)=S_{k_i(m_1+1)}(\hat z, x_i),
\end{split}
\end{equation*}
we find
\begin{equation}\label{d5}
\frac{S_{k_i(m_1+1)}(\hat z, x_i)}{S_{k_im_1}(\hat z, x_i)}=e^{ik_i\hat\theta}\frac{S_{k_i(m_1+1)}(\hat z_1, x_i)}{S_{k_im_1}(\hat z_1, x_i)}.
\end{equation}
From (\ref{d5}) we have
\begin{equation*}
\mathcal R_{k_im_1}(\hat z)+i\mathcal I_{k_im_1}(\hat z)=(\cos k_i\hat \theta+i\sin k_i\hat\theta)(\mathcal R_{k_im_1}(\hat z_1)+i\mathcal I_{k_im_1}(\hat z_1))
\end{equation*}
which leads to
\begin{equation*}
\mathcal R_{k_im_1}(\hat z)
=\mathcal R _{k_im_1}(\hat z_1)\cos
k_i\hat\theta-\mathcal I_{k_im_1}(\hat z_1)\sin
k_i\hat\theta.
\end{equation*}

Then we get
\begin{equation}\label{d6}
1-\mathcal R_{k_im_1}(\hat z)=1+(1-\mathcal
R_{k_im_1}(\hat z_1))\cos k_i\hat\theta-\cos k_i\hat\theta+\mathcal
I_{k_im_1}(\hat z_1)\sin k_i\hat\theta.
\end{equation}
From (\ref{d6}) we have
\begin{equation*}
|1-\mathcal R_{k_im_1}(\hat z)|\geq 1-\cos k_i\hat\theta-|1-\mathcal R_{k_im_1}(\hat z_1)|-|\mathcal I_{k_im_1}(\hat z_1)|.
\end{equation*}
By (3) in Proposition \ref{p1} we have
\begin{equation*}
\sigma\geq|1-\mathcal R_{k_im_1}(\hat z)|\geq 1-(1-\delta_0)-\sigma-\frac{\sigma}{8},
\end{equation*}
that is
\begin{equation*}
(2+\frac18)\sigma\geq\delta_0.
\end{equation*}
This is contradiction with $0<\sigma<\frac{\delta_0}{100}$. Thus we get the conclusion of the claim.

From the above claim and  by (4) in Proposition \ref{p1}, we have
\begin{equation*}
|S_{k_im_2}(\hat z, x_i)|>\frac{c}{2}, |S_{k_im_2}(\hat y, x_i)|<\frac{c}{8}.
\end{equation*}
This is a contradiction with
\begin{equation*}
S_{k_im_2}(\hat z, x_i)=S_{k_im_2}(\hat y, x_i).
\end{equation*}

Next, we assume $\hat y, \hat z\in X_{\rm reg}$.

Case III: $\hat y, \hat z\in X_{\rm reg}$ and $\hat y=e^{i\hat \theta}\circ\hat z$ for some $\hat\theta\in [0, 2\pi)$. Choose canonical  coordinates $(z, \theta, \varphi)$ centered at some $\hat z_0\in X$ and defined on a canonical local patch $D=\{(z, \theta): |z|<\varepsilon, |\theta|<\pi\}$ such that in terms of the canonical coordinates, $\hat z=(0, \theta_1)$, $\hat y=(0, \theta_2).$ Then we have $\theta_2-\theta_1=\hat\theta$. Let $\{f_j\}_{j=1}^{d_m}$ and $\{g_j\}_{j=1}^{d_{m+1}}$ be orthonormal bases of $H^0_{b, m}(X)$ and $H^0_{b, m+1}(X)$, respectively. Then by the assumptions on $\hat y_m$ and  $\hat z_m$ we have
\begin{equation*}
\begin{split}
S_m(\hat z_m, \hat y_m)=S_m(\hat z_m, \hat z_m),
S_{m+1}(\hat z_m, \hat y_m)=S_{m+1}(\hat z_m, \hat z_m).
\end{split}
\end{equation*}
Without loss of generality, we assume $\hat z_m, \hat y_m\in D$ for each $m$. Then in terms of canonical local coordinates, we write $\hat z_m=(z_m, \theta_m)$ and $\hat y_m=(w_m, \eta_m).$ By Theorem \ref{c2} we have \begin{equation}\label{d7}
\begin{split}
S_m(\hat z_m, \hat y_m)&=\frac{1}{2\pi}e^{im(\theta_m-\eta_m+\Phi(z_m, w_m))}\hat b(z_m, w_m, m)+O(m^{-\infty}),\\
S_{m+1}(\hat z_m, \hat y_m)&=\frac{1}{2\pi}e^{i(m+1)(\theta_m-\eta_m+\Phi(z_m, w_m))}\hat b(z_m, w_m, m+1)+O(m^{-\infty}),\\
S_m(\hat z_m, \hat z_m)&=\frac{1}{2\pi}\hat b(z_m, z_m, m)+O(m^{-\infty}),\\
S_{m+1}(\hat z_m, \hat z_m)&=\frac{1}{2\pi}\hat b(z_m, z_m, m+1)+O((m+1)^{-\infty}).
\end{split}
\end{equation}

We assume $\lim\limits_{m\rightarrow\infty}m{\rm Im}\Phi(z_m, w_m)=M$ ($M$ can be $\infty$).

(a): Assume
\begin{equation*}
\lim\limits_{m\rightarrow\infty}m{\rm Im}\Phi(z_m, w_m)=M\in (0, \infty].
\end{equation*}
From $S_m(\hat z_m, \hat y_m)=S_m(\hat z_m, \hat z_m)$ and (\ref{d7}) we have
\begin{equation*}
e^{im(\theta_m-\eta_m+\Phi(z_m, w_m))}\hat b(z_m, w_m, m)=\hat b(z_m, z_m, m)+O(m^{-\infty}).
\end{equation*}
Then we get
\begin{equation*}
m^{-(n-1)}|\hat b(z_m, w_m, m)|e^{-m{\rm Im}\Phi(z_m, w_m)}=m^{-(n-1)}|\hat b(z_m, z_m, m)+O(m^{-\infty})|.
\end{equation*}
Letting $m\rightarrow\infty$, we have
\begin{equation*}
\hat b(0, 0)= e^{-M}\hat b(0, 0),
\end{equation*}
that is $\hat b(0, 0)=0$. Thus, we get a contradiction.

(b): Assume
\begin{equation}\label{d8}
\lim\limits_{m\rightarrow\infty}m{\rm Im}\Phi(z_m, w_m)=0.
\end{equation}
From $S_{m+1}(\hat z_m, \hat y_m)-S_m(\hat z_m, \hat y_m)=S_{m+1}(\hat z_m, \hat z_m)-S_{m}(\hat z_m, \hat z_m)$ combined with (\ref{d7}) we have
\begin{equation*}
\begin{split}
&m^{-(n-1)}\abs{e^{im(\theta_m-\eta_m+\Phi(z_m, w_m))}\left[e^{i(\theta_m-\eta_m+\Phi(z_m, w_m))}\hat b(z_m, w_m, m+1)-\hat b(z_m, w_m, m)\right]}\\
=&m^{-(n-1)}\abs{\hat b(z_m, z_m, m+1)-\hat b(z_m, z_m, m)}+O(m^{-\infty}).
\end{split}
\end{equation*}
Letting $m\rightarrow\infty$ and using (\ref{d8}), we have
\begin{equation*}
|e^{i\hat\theta}\hat b(0, 0)-\hat b(0, 0)|=0.
\end{equation*}
Hence, $\hat\theta=0$ and $\hat z=\hat y.$ Put
\begin{equation*}
f_m(t)=\frac{|S_m(t\hat z_m+(1-t)\hat y_m, \hat y_m)|^2}{S_m(t\hat z_m+(1-t)\hat y_m,t\hat z_m+(1-t)\hat y_m)S_m(\hat y_m,\hat y_m)}.
\end{equation*}
We have
\begin{equation}\label{d9}
\begin{split}
&f_m(0)=\frac{S_m(\hat y_m,\hat y_m)^2}{S_m(\hat y_m,\hat y_m)^2}=1,\\
&f_m(1)=\frac{\abs{S_m(\hat z_m, \hat y_m)}^2}{S_m(\hat z_m,\hat z_m)S_m(\hat y_m,\hat y_m)}=\frac{S_m(\hat y_m, \hat y_m)^2}{S_m(\hat y_m,\hat y_m)S_m(\hat y_m,\hat y_m)}=1.
\end{split}
\end{equation}
By the Schwartz inequality, we get $0\leq f_m(t)\leq 1.$ Then (\ref{d9}) implies that there is a $t_m\in (0, 1)$ such that
$
f'_m(t_m)=0$ and $f_m^{''}(t_m)\geq 0
$
holds. Hence, we get
\begin{equation}\label{e1}
\liminf_{m\rightarrow\infty} \frac{f_m^{''}(t_m)}{|z_m-w_m|^2m}\geq0.
\end{equation}
Then, making use of the same arguments as in \cite{H14a}((4.22) in Theorem 4.7), we have that (\ref{e1}) is impossible under the assumption (\ref{d8}).

Case IV: $\hat z, \hat y\in X_{\rm reg}$, $\hat y\neq e^{i\theta}\circ\hat z$ for any $\theta\in[0, 2\pi).$ Choose a canonical local patch $D(\hat z)$  with canonical coordinates $(z, \theta, \varphi)$ centered at $\hat z$. Since $\hat z\in X_{\rm reg}$, we can apply Lemma \ref{l-gue150615} and have that $D(\hat z)$ can be chosen such that $D(\hat z)=\{(z, \theta): |z|<\varepsilon, |\theta|<\pi\}$ holds in terms of canonical coordinates. Let $\varepsilon$ be sufficiently small such that $\overline{ D(\hat z)}$
 is $S^1$-invariant.
Since $\hat y\neq e^{i\theta }\circ \hat z$ for all $\theta\in [0, 2\pi)$, for $\varepsilon $ small enough we can choose a canonical local patch $D(\hat y)$ such that $\overline{D(\hat y)}\cap \overline{D(\hat z)}=\emptyset$ holds. Choose two functions $\chi, \chi_1\in C_0^{\infty}(X)$ satisfying $\chi=1$ in a small neighborhood of $\overline{D(\hat z)}$ and ${\chi_1}=1$ in a small neighborhood of ${\rm supp}\chi$ and ${\rm supp}\chi\cap \overline{D(\hat y)}={\rm  supp}\chi_1\cap \overline{D(\hat y)}=\emptyset.$ Choose $\chi_0(w)\in C_0^\infty(\mathbb C^{n-1})$ such that ${\rm supp}\chi_0(w)\Subset \{w: |w|<1\}$ and $\int_{\mathbb C^{n-1}}\chi_0(w)dv(w)=1$ hold. Furthermore, choose $\eta_0(y_{2n-1})\in C_0^\infty(-\pi, \pi)$ with $\int_{-\pi}^\pi\eta_0(y_{2n-1})dy_{2n-1}=1$. For any $m\in\mathbb N$, set
\begin{equation}
u_m(y)=m^{n-1}e^{im (y_{2n-1}-\theta_m-{\rm Re}\Phi(z_m, w))}\eta_0(y_{2n-1})\chi_0(m(w-z_m))\in C_0^\infty(D(\hat z)).
\end{equation}

Then we have \begin{equation}\label{f6}
S_mu_m(\hat y_m)=\chi S_mu_m(\hat y_m)+(1-\chi)S_mu_m(\hat y_m)=(1-\chi)S_mu_m(\hat y_m)
\end{equation}
and
\begin{equation}\label{f7}
(1-\chi) S_mu_m(\hat y_m)=(1-\chi)S\chi_1Q_mu_m(\hat y_m)+(1-\chi)S(1-\chi_1)Q_mu_m(\hat y_m).
\end{equation}
Since $\overline{D(\hat z)}$ is an $S^1$-invariant subset and ${\rm supp} u_m\Subset D(\hat z)$, we have ${\rm supp} Q_mu_m\Subset \overline{D(\hat z)}$. This implies
\begin{equation}\label{f8}
(1-\chi)S(1-\chi_1)Q_mu_m(\hat y_m)=0.
\end{equation}
Then by the same arguments as in the proof of Lemma \ref{f5}, we get
\begin{equation}\label{f9}
(1-\chi)S\chi_1Q_mu_m(\hat y_m)=O(m^{-\infty}).
\end{equation}
Combining (\ref{f6}), (\ref{f7}), (\ref{f8}) and (\ref{f9}), we conclude
\begin{equation*}
S_mu_m(\hat y_m)=O(m^{-\infty}).
\end{equation*}
On the other hand, we have
\begin{equation*}
\begin{split}
S_mu_m(\hat z_m)
=&\frac{m^{n-1}}{2\pi}\int_X e^{-m{\rm Im}\Phi(z_m, w)}\hat b(z_m, w, m)\chi_0(m(w-z_m))\lambda(w)dv(w)+O(m^{-\infty})\\
=&\frac{1}{2\pi}\int_{\{w\in\mathbb C^{n-1}: |w|<1\}}e^{-m{\rm Im}\Phi(z_m, \frac{w}{m}+z_m)}\hat b(z_m, \frac{w}{m}+z_m, m)\times\\
&\chi_0(w)\lambda(\frac{w}{m}+z_m)m^{-(n-1)}dv(w)
+O(m^{-\infty}).
\end{split}
\end{equation*}
Since ${\rm Im}\Phi(z_m, \frac{w}{m}+z_m)\geq c_0|\frac{w}{m}|^2$ for some constant $c_0$, we have $-m{\rm Im}\Phi(z_m, \frac{w}{m}+z_m)\rightarrow0$  uniformly on $\{w\in\mathbb C^{n-1}: |w|<1\}$ as $m\rightarrow\infty$.
Letting $m\rightarrow\infty$ we have
\begin{equation*}
\lim\limits_{m\rightarrow\infty}S_mu_m(\hat z_m)=\frac{1}{2\pi}\hat b(0, 0)\neq 0.
\end{equation*}
This is a contradiction with the assumption  $S_mu_m(\hat z_m)=S_mu_m(\hat y_m)$.

Case V: $\hat z\in X_{\rm reg}, \hat y\not\in X_{\rm reg}.$ We have that $\hat y\neq e^{i\theta}\circ \hat z$ for all $\theta\in[0, 2\pi).$ Following the same arguments as in Case IV, we find that this is impossible.

Thus, we get the conclusion of Theorem \ref{main1}.
\end{proof}

\begin{center}
{\bf Acknowledgement}
\end{center}

The authors would like to thank Professor Paul Yang for his interest in this work. The third-named author would like to thank the Institute of Mathematics, University of Cologne for the hospitality during his visit. We also thank the referee for many detailed remarks that have helped to improve the presentation.

\bibliographystyle{amsalpha}

\end{document}